\documentclass[11pt]{amsart}

\usepackage{fullpage}
\usepackage{verbatim}
\usepackage{amssymb}
\usepackage{amscd}
\usepackage{amsmath}
\usepackage{eucal}
\usepackage{setspace}
\usepackage{url}
\usepackage[utf8]{inputenc}

\newtheorem{theorem}{Theorem}[section]

\newtheorem{corollary}[theorem]{Corollary}
\newtheorem{lemma}[theorem]{Lemma}
\newtheorem*{conjecture}{Conjecture}
\newtheorem*{hypothesis}{Hypothesis}
\newtheorem{proposition}[theorem]{Proposition}
\numberwithin{equation}{subsection}

\theoremstyle{remark}

\newtheorem{remark}[theorem]{Remark}

\renewcommand{\leq}{\leqslant}
\renewcommand{\geq}{\geqslant}
\newcommand{\T}[1]{{}^t{{#1}}}

\newcommand{\Q}{\mathbb Q}

\renewcommand{\H}{\mathbb H}

\newcommand{\Z}{\mathbb Z}
\newcommand{\Pp}{\mathbb P}
\newcommand{\C}{\mathbb C}
\newcommand{\R}{\mathbb R}

\newcommand{\A}{\mathbb A}
\newcommand{\bs}{\backslash}

\newcommand{\B}{\mathcal B}
\newcommand{\p}{\mathfrak p}

\renewcommand{\L}{\mathcal L}
\newcommand{\Tr}{\mathrm{Tr}}

\newcommand{\Sp}{\mathrm{Sp}}
\newcommand{\GSp}{\mathrm{GSp}}
\newcommand{\USp}{\mathrm{USp}}
\newcommand{\GL}{\mathrm{GL}}
\newcommand{\SO}{\mathrm{SO}}
\newcommand{\SL}{\mathrm{SL}}
\newcommand{\USpin}{\mathrm{USpin}}
\newcommand{\Spin}{\mathrm{Spin}}
\newcommand{\cusp}{\mathcal{S}}
\newcommand{\classical}{\mathcal{H}}
\newcommand{\field}{\mathcal{K}}
\newcommand{\weight}{\omega}

\DeclareMathOperator{\Imag}{Im}
\DeclareMathOperator{\Reel}{Re}
\DeclareMathOperator{\Cl}{Cl}
\DeclareMathOperator{\disc}{disc}
\DeclareMathOperator{\Aut}{Aut}

\newcommand{\sop}{\uple{S}}

\newcommand{\eps}{\varepsilon}
\newcommand{\proba}{\mathbf{P}}
\newcommand{\expect}{\mathbf{E}}

\newcommand{\mods}[1]{\,(\mathrm{mod}\,{#1})}

\newcommand{\ra}{\rightarrow}

\newcommand{\uple}[1]{\text{\boldmath${#1}$}}

\begin{document}

\title[Local equidistribution for Siegel modular forms] {Local
  spectral equidistribution for Siegel modular forms and applications}

\author{Emmanuel Kowalski}
\address{ETH Z\"urich -- D-MATH\\
  R\"amistrasse 101\\
  8092 Z\"urich\\
  Switzerland} \email{kowalski@math.ethz.ch}
\author{Abhishek Saha}
\address{ETH Z\"urich -- D-MATH\\
  R\"amistrasse 101\\
  8092 Z\"urich\\
  Switzerland} \email{abhishek.saha@math.ethz.ch}
\author{Jacob Tsimerman}
\address{Princeton University Maths. Dept
\\ Fine Hall\\ Washington Road, Princeton
  NJ 08540, USA}
\email{jtsimerm@math.princeton.edu}



\begin{abstract}
  We study the distribution, in the space of Satake parameters, of
  local components of Siegel cusp forms of genus $2$ and growing
  weight $k$, subject to a specific weighting which allows us to apply
  results concerning Bessel models and a variant of Petersson's
  formula. We obtain for this family a quantitative local
  equidistribution result, and derive a number of consequences. In
  particular, we show that the computation of the density of low-lying
  zeros of the spinor $L$-functions (for restricted test functions)
  gives global evidence for a well-known conjecture of B\"ocherer
  concerning the arithmetic nature of Fourier coefficients of Siegel
  cusp forms.
\end{abstract}

\bibliographystyle{plain}

\maketitle

\section{Introduction}

\subsection{Motivation}

The motivation behind this paper lies in attempts to understand what
is a correct definition of a family of cusp forms, either on $\GL(n)$
or on some other reductive algebraic group.  The basic philosophy (or
strategy) underlying our work is the following form of local-global
principle: given a ``family'' $\Pi$ of cusp forms, for any finite
place $v$, the local components $\pi_v$ of the elements $\pi \in \Pi$
(which are represented as restricted tensor products $\pi=\otimes
\pi_v$ over all places) should be well-behaved, and more specifically,
under averaging over finite subsets of the family, $(\pi_v)$ should
become equidistributed with respect to a suitable measure
$\mu_v$. Readers already familiar with work on families of Dirichlet
characters, classical modular forms on $\GL(2)$, or families of
$L$-functions of abelian varieties, will recognize that this principle
is implicit in much of these works, through the orthogonality
relations for Dirichlet characters, the trace formula (or the
Petersson formula) and the ``vertical'' Sato-Tate laws over finite
fields. The expected outcome is that, for instance, averages over the
family of values of $L$-functions $L(s_0,\pi)$ at some point $s_0$, at
least on the right of the critical line, should be directly related to
the Euler product corresponding to local averages computed using
$\mu_v$. (For a general informal survey of this point of view,
see~\cite{ksurvey}).
\par
One of our goals is to give an example where this strategy can be
implemented in the case of holomorphic cusp forms on $\GSp(4)$ (i.e.,
Siegel modular forms), and to derive some applications of it.  In
particular, we will prove:\footnote{\ Unfamiliar notation will be
  explained later.}

\begin{theorem}\label{th-average-l-functions}
  For $k\geq 2$, let $\cusp_k^*$ be a Hecke basis of the space of
  Siegel cusp forms on $\Sp(4,\Z)$, and let $\cusp_k^{\flat}$ be the
  set of those $F\in\cusp_k^*$ which are \emph{not} Saito-Kurokawa
  lifts. For $F\in \cusp_k^*$, let
$$
F(Z)=\sum_{T>0}a(F,T)e(\Tr (TZ))
$$
be its Fourier expansion, where $T$ runs over symmetric positive-definite
semi-integral  matrices, and let
\begin{equation}\label{e:weight}
\weight^F_{k} = \sqrt{\pi} (4 \pi)^{3-2k} \Gamma(k -
\tfrac{3}{2})\Gamma(k - 2) \frac{|a(F,1)|^2}{4\langle F, F\rangle}.
\end{equation}
\par
Let $L(F,s)$ denote the finite part of the spin $L$-function of $F$,
an Euler product of degree $4$ over primes with local factors of the
form
$$
L_p(F,s)=(1-a_pp^{-s})^{-1} (1-b_pp^{-s})^{-1} (1-a_p^{-1}p^{-s})^{-1}
(1-b_p^{-1}p^{-s})^{-1},\quad\quad a_p,\ b_p\in (\C^{\times})^2.
$$
\par
Then, for any $s\in\C$ such that $\Reel(s)>1$, we have
\begin{equation}\label{e:firstaverage}
  \lim_{k\ra +\infty }{\sum_{F\in \cusp_{2k}^{\flat}}
{\weight^F_{2k}L(F,s)}} = \zeta(s +
  \tfrac{1}{2})L(\chi_4, s+\tfrac{1}{2}),
\end{equation}
where $\zeta(s)$ denotes the Riemann zeta function and $L(\chi_4, s)$
is the $L$-function associated to the unique Dirichlet character of
conductor $4$. For $s\not=3/2$, one can replace $\cusp_{2k}^{\flat}$
with $\cusp_{2k}^*$.
\par
More generally, for all primes $p$ there exist measures $\mu_p$ on
$(\C^\times)^2$, which are in fact supported on $(S^1)^2$, with the
following property: for any irreducible $r$-dimensional representation
$\rho$ of $\GSp(4,\C)$, let $L(F,\rho, s)$ denote the associated
Langlands $L$-function, an Euler product of degree $r\geq 1$ over
primes with local factors of the form
$$
L_p(F,\rho,s)=\prod_{i=1}^r{(1- Q_{i}(a_p, b_p)p^{-s})^{-1}}
,\quad\quad a_p,\ b_p\in (\C^{\times})^2.
$$
where $Q_{i}(x,y)$ is a polynomial in $x$, $y$, $x^{-1}$,
$y^{-1}$. Then for any $s\in\C$ such that $\Reel(s)>s_0$, with $s_0$
depending on $\rho$, we have
\begin{equation}\label{e:general-rho}
\lim_{k\ra +\infty}{\sum_{F\in \cusp_{2k}^*}{\weight^F_{2k}L(F,\rho, s)}}
=\prod_{p}{
\int{
\prod_{i=1}^r(1- Q_{i}(a, b)p^{-s})^{-1}d\mu_p(a,b)}
},
\end{equation}
where the right-hand side converges absolutely.
\end{theorem}

The weight $\weight^F_{k}$ which is introduced in this theorem is
natural because of our main tool, which is a (rather sophisticated)
extension of the classical Petersson formula to the case of Siegel
modular forms of genus 2; see Propositions~\ref{proppetersson2}
and~\ref{pr:petersson}. In fact, we can work with more general weights
$\weight^F_{k,d,\Lambda}$ (as defined in the next section) which
involve averages of $a(F,T)$ over positive definite $T$ with a fixed
discriminant. 
\par
The quantitative local equidistribution leads naturally to a result on
the distribution of low-lying zeros:

\begin{theorem}[Low-lying zeros]\label{th:low-lying}
  Let $\varphi\,:\, \R\ra \R$ be an even Schwartz function such that
  the Fourier transform
$$
\hat{\varphi}(t)=\int_{\R}{\varphi(x)e^{-2i\pi xt}dx}
$$
has compact support contained in $[-\alpha,\alpha]$, where
$\alpha<4/15$. For $F\in \cusp_{2k}^*$, assume the Riemann Hypothesis:
the zeros of $L(F,s)$ in the critical strip $0<\Reel(s)<1$ are of the
form
$$
\rho=\frac{1}{2}+i\gamma
$$
with $\gamma\in\R$. Define 
$$
D_{\varphi}(F)=\sum_{\rho}{ \varphi\Bigl(\frac{\gamma}{\pi}\log
  k\Bigr) },
$$
where $\rho$ ranges over the zeros of $L(F,s)$ on the critical line,
counted with multiplicity.
\par
Then we have
\begin{equation}\label{eq:mock-unitary}
\lim_{k\ra +\infty}
\sum_{F\in\cusp_{2k}^*}{
\omega_{2k}^{F}D_{\varphi}(F)
}
=\int_{\R}{\varphi(x)d\sigma_{Sp}(x)},
\end{equation}
where $\sigma_{Sp}$ is the ``symplectic symmetry'' measure given by
$$
d\sigma_{Sp}=dx-\frac{\delta_0}{2},\quad
\quad
\text{ $\delta_0$ Dirac mass at 0}.
$$
\end{theorem}

This result raises interesting questions concerning the notion of
``family'' of cusp forms, especially from the point of view of the
notion of symmetry type that has arisen from the works of
Katz-Sarnak~\cite{ks}. Indeed, the limit measure above is the one that
arises from symplectic symmetry types, i.e., from the distribution of
eigenvalues close to $1$ of symplectic matrices of large size, when
renormalized to have averaged spacing equal to $1$. In general, it is
expected that some cusp forms will exhibit this symmetry when some
kind of infinite symplectic group occurs as ``monodromy group'' for
the family, in the way that generalizes the Chebotarev and Deligne
equidistribution theorems. 
\par
\emph{We do not believe} that this is the case here, and rather expect
that the limit measure in the theorem is due in part to the presence
of the weight $\omega_{2k}^{F}$ used in the averages
involved. Precisely, we expect that the correct symmetry type, without
weight, is \emph{orthogonal}, in the sense that for $\varphi$ with
support in $]-1,1[$, we should have\footnote{\ We do not try to
  predict whether odd or even orthogonal symmetry should occur; these
  could be distinguished most simply by computing the $2$-level
  density for test functions with restricted support as done
  in~\cite{miller} for classical modular forms (one can also attempt
  to study the low-lying zeros for test functions with support larger
  than $]-1,1[$, as done in~\cite{ils}.}
\begin{equation}\label{eq-orth-symm}
\frac{1}{|\cusp_{2k}^*|}
\sum_{F\in\cusp_{2k}^*}D_{\varphi}(F)
\longrightarrow \int_{\R}{\varphi(x)d\sigma_{O}(x)}
\end{equation}
where $d\sigma_O(x)=dx+\frac{\delta_0}{2}$ is the corresponding
measure for eigenvalues close to $1$ of orthogonal matrices.
\par
Intuitively, this should be related to the fact that the point $1/2$
is a critical special value -- in the sense of Deligne -- for the spin
$L$-functions of cusp forms $F\in\cusp_{k}^*$ (whereas $1/2$ is not
for real quadratic characters for example, which are the typical
example where symplectic symmetry is expected), similar to the special
role of the eigenvalue $1$ for orthogonal matrices, but not for
symplectic ones.
\par
Now the natural question is why should the weight $\omega_{2k}^F$ have
such an effect? (This is especially true because it may look, at
first, just like an analogue of the weight involving the Petersson
norm of classical modular forms which has been used very frequently
without exhibiting any such behavior, e.g., in the works of Iwaniec,
Luo and Sarnak~\cite{ils} and Duenez-Miller~\cite{duenez-miller}.)
\par
The point is that this weight $\weight_k^F$ itself \emph{contains
  arithmetic information related to central $L$-values of the Siegel
  cusp forms}. Indeed, we will see in Section~\ref{sec:lowlying} that
Theorem~\ref{th:low-lying} can be interpreted convincingly -- assuming
an orthogonal symmetry as in~(\ref{eq-orth-symm}) -- as evidence for a
beautiful conjecture of B\"ocherer (see~\cite{boch-conj}
or~\cite[Intr.]{furusawa-shalika-fund}) which suggests in particular a
relation of the type
\begin{equation}\label{furmart}
|a(F,1)|^2 \simeq L(F,1/2)L(F\times \chi_4,1/2)
\end{equation}
(where the $\simeq$ sign means equality up to non-zero factors
``unrelated to central critical values''; this version of the
conjecture is that proposed by Furusawa and Martin~\cite[\S 1,
(1.4)]{furusawa-martin}). We therefore consider that
Theorem~\ref{th:low-lying} provides suggestive global evidence towards
these specific variants of B\"ocherer's conjecture. Note that, at the
current time, this conjecture is not rigorously known for any cusp
form in $\cusp_{2k}^*$ which is not a Saito-Kurokawa lift.
\par
\begin{remark}
  One can easily present analogues of the phenomenon in
  Theorem~\ref{th:low-lying}, as we understand it, in the setting of
  random matrices. For instance, if $\mu_n$ denotes the probability
  Haar measure on $\SO_{2n}(\R)$, one may consider the measures
$$
d\nu_n(g)=c_n\det(1-g)d\mu_n(g),
$$
where $c_n>0$ is the constant that ensures that $\nu_n$ is a
probability measure.\footnote{\ In fact, although this seems
  irrelevant, a computation with the moments of characteristic
  polynomials of orthogonal matrices shows that $c_n=1/2$ for all
  $n$.} The distribution of the low-lying eigenvalues of $g\in
\SO_{2n}(\R)$, when computed using this measure, will clearly differ
from that arising from Haar measure (intuitively, by diminishing the
influence of matrices with an eigenvalue close to $1$, the factor
$\det(1-g)$ will produce a repulsion effect similar to what happens
for symplectic matrices.)
\end{remark}

Readers familiar with the case of $\GL(2)$-modular forms but not with
Siegel modular forms (or with their representation-theoretic
interpretation) may look at the Appendix where we discuss briefly the
analogies and significant differences between our results and some
more elementary $\GL(2)$-versions.
\par
For orientation, we add the following quick remarks: (1) the spin
$L$-function of $F\in\cusp_{2k}^*$ has analytic conductor
(see~\cite[p. 95]{ant} for the definition) of size $k^2$; (2) the
cardinality of $\cusp_{2k}^*$ (i.e., the dimension of $\cusp_{2k}$) is
of order of magnitude $k^3$ (see, e.g.,~\cite[Cor. p. 123]{klingen}
for the space $\mathcal{M}_{2k}$ of all Siegel modular forms of weight
$2k$, and~\cite[p. 69]{klingen} for the size of the ``correction
term'' $\mathcal{M}_{2k}/\cusp_{2k}$); (3) as already mentioned, the
spin $L$-function is self-dual with functional equation involving the
sign $+1$ for all $F\in \cusp_{2k}^*$.
\par
Apart from the treatment of low-lying zeros, we do not ``enter the
critical strip'' in this paper.  However, we hope to come back to the
problem of extending Theorem~\ref{th-average-l-functions} to averages
at points inside the critical strip, and we may already remark that,
if a statement like~\eqref{furmart} is valid, the weight already
involves some critical values of $L$-functions (in fact, of an
$L$-function of degree $8$. 

\subsection{Local equidistribution statement}
\label{sec:statement}

In order to state our main result on local spectral equidistribution of
Siegel modular forms, we begin with some preliminary notation
concerning cuspidal automorphic representations of $G(\A)=\GSp(4,\A)$.
\par
Let $\pi$ be a cuspidal automorphic representation of $G(\A)$, which
we assume to be unramified at all finite places and with trivial
central character. It is isomorphic to a restricted tensor product
$\pi = \otimes_v \pi_{v}$ where, for all places, $\pi_v$ is an
irreducible admissible unitary representation of the local group
$G(\Q_v)$. 
\par
By our assumption $\pi_{p}$ is unramified for all primes $p$ and so
the natural underlying space for local equidistribution at $p$ (when
considering families) is the set $X_p$ of unramified unitary
infinite-dimensional irreducible representations of $G(\Q_p)$ with
trivial central character. This set has a natural topology, hence a
natural $\sigma$-algebra. \par

We now proceed quite concretely to give
natural coordinates on $X_p$ from which the measurable structure is
obvious.
By~\cite{cart}, any $\pi_p\in X_p$ can be identified with the unique
unramified constituent of a representation $\chi_1 \times \chi_2
\rtimes \sigma$ induced from a character of the Borel subgroup which
is defined as follows using unramified (not necessarily unitary!)
characters $\chi_1$, $\chi_2$, $\sigma$ of $\Q_p^\times$:
$$
\begin{pmatrix}
a_1&\ast&\ast&\ast\\&a_2&\ast&\ast\\&&\lambda
  a_1^{-1}&\\&&\ast&\lambda a_2^{-1}
\end{pmatrix} \mapsto
\chi_1(a_1)\chi_2(a_2)\sigma(\lambda).
$$
\par
Having trivial central character means that
$$
\chi_1 \chi_2 \sigma^2 =1,
$$
and since the characters are unramified, it follows that $\pi_p$ is
characterized by the pair $(a,b)=(\sigma(p),\sigma(p)\chi_1(p))\in
\C^*\times\C^*$. The classification of local representations of
$G(\Q_p)$ (see for instance~\cite[Proposition 3.1]{pitschram}) implies
that the local parameters satisfy
\begin{equation}\label{eq:unitarity}
0<|a|,\ |b|\leq \sqrt{p}.
\end{equation}
\par
There are some identifications between the representations associated
to different $(a,b)$, coming from the Weyl $W$ group of order $8$
generated by the transformations
\begin{equation}\label{eq-weyl-group}
(a,b)\mapsto (b,a),\quad (a,b)\mapsto (a^{-1},b),\quad (a,b)\mapsto (a,b^{-1}).
\end{equation}
\par
We will denote by $Y_p$ the quotient of the set of $(a,b)$ satisfying
the upper-bounds~(\ref{eq:unitarity}), modulo the action of $W$. This
has the quotient topology and quotient $\sigma$-algebra, and we
identify $X_p$ with a subset of $Y_p$ using the parameters $(a,b)$
described above.  We will also denote by $X\subset X_p$ the subset of
tempered representations; under the identification of $X_p$ with a
subset of $Y_p$, the set $X$ corresponds precisely to
$|a|=|b|=1$. Note that this subset is indeed independent of $p$.
\par
In applications to $L$-functions, the local-global nature of
automorphic representations is reflected not only in the existence of
local components, but in their ``independence'' (or product structure)
when $p$ varies. To measure this below, we will also need to consider,
for any finite set of primes $\sop$, the maps
$$
\pi\mapsto (\pi_p)_{p\in \sop}
$$
which have image in the space
$$
X_{\sop}=\prod_{p\in \sop}{X_p}
$$
and can be identified with a subset of
$$
Y_{\sop}=\prod_{p\in \sop}{Y_p}.
$$
\par
Now we come back to Siegel modular forms. Let
$\cusp_k=\cusp_k(\Sp(4,\Z))$ be the space of Siegel cusp forms of
degree $2$, level $1$ and weight $k$. By ad\'elization (as described
in more detail in the next section), there is a cuspidal automorphic
representation $\pi_F$ canonically attached to $F$; the assumption
that the level is $1$ and there is no nebentypus means that $\pi_F$ is
unramified at finite places with trivial central character, as
above. Thus we have local components $\pi_p(F)\in X_p$ and
corresponding parameters $(a_p,b_p)\in Y_p$ for every prime $p$.

The generalized Ramanujan conjecture has been proved in this setting
by Weissauer~\cite{weissram}: it states that, if $F$ is not a
Saito-Kurokawa lift (these forms are defined in~\cite{eichzag} for
example; at the beginning of Section~\ref{ssec-averaging}, we recall
the description in terms of $L$-functions), we have $\pi_p(F)\in X$
for all $p$, i.e., $|a_p|=|b_p|=1$. On the other hand, if $F$ is a
Saito-Kurokawa lift, then $|a|=1$ and
$$
\{|b|,|b|^{-1}\}=\{p^{1/2},p^{-1/2}\}.
$$

\begin{remark}
  Partly because our paper is meant to explore the general philosophy
  of families of cusp forms, we will not hesitate to use this very
  deep result of Weissauer when this helps in simplifying our
  arguments. But it will be seen that the proof of the local
  equidistribution property itself does not invoke this result, and it
  seems quite likely that, with some additional work, it could be
  avoided in most, if not all, of the applications (in similar
  questions of local equidistribution for classical Maass cusp forms
  on $\GL(2)$, one can avoid the unproved Ramanujan-Petersson
  conjecture).
\end{remark}

Denote by $\cusp_k^*$ any fixed Hecke-basis of $\cusp_k$. Although
this is not known to be unique, the averages we are going to consider
turn out to be independent of this choice. In fact, all final results
could be phrased directly in terms of automorphic representations,
avoiding such a choice (at least seemingly).

We next proceed to define our way of weighting the cusp forms in
$\cusp_k^*$. This generalizes the $\weight_k^F$ in the statement of
the first theorem, and the reader may assume below that the parameters
introduced are $d=4$ and $\Lambda=1$.

Let $d>0$ be a positive integer such that $-d$ is a fundamental
discriminant of an imaginary quadratic field (i.e., one of the
following holds: (1) $d$ is congruent to $3 \pmod{4}$ and is
square-free; or (2) $d= 4m$ where $m$ is congruent to $1$ or $2
\pmod{4}$ and $m$ is square-free).  Let $\Cl_d$ denote the ideal class
group of this field, $h(-d)$ denote the class number, i.e., the
cardinality of $\Cl_d$, and $w(-d)$ denote the number of roots of
unity. Finally, fix a character $\Lambda$ of $\Cl_d$.

There is a well-known natural isomorphism, to be recalled more
precisely in Section~\ref{s:classical}, between $\Cl_d$ and the
$\SL(2,\Z)$-equivalence classes of primitive semi-integral two by two
positive definite matrix with determinant equal to $d/4$.  By abuse of
notation, we will also use $\Cl_d$ to denote the set of equivalence
classes of such matrices.

Define normalizing factors
$$
c_{k,d} =
\bigg(\frac{d}{4}\bigg)^{\frac{3}{2}-k}\frac{4c_k}{w(-d)h(-d)},
$$
where 
$$
c_k =  \frac{\sqrt{\pi}}{4}(4 \pi)^{3 -2k} 
\Gamma(k - \tfrac{3}{2})\Gamma(k-2).
$$
\par
We note that, using Dirichlet's class number formula, one can also
write
$$c_{k,d} = \bigg(\frac{d}{4}\bigg)^{1-k} \frac{4\pi
  c_k}{w(-d)^2 L(1, \chi_d)}
$$ 
where $\chi_d$ is the real primitive Dirichlet character associated to
the extension $\Q(\sqrt{-d})$.
Now, for each $F\in \cusp_k^*$, we have a Fourier
expansion
$$
F(Z)=\sum_{T>0}a(F,T)e(\Tr (TZ)),
$$
where $T$ runs over positive definite symmetric semi-integral matrices of
size $2$:
$$
T=\begin{pmatrix}a&b/2\\
b/2&c
\end{pmatrix}
$$
with $(a,b,c)\in \Z$. It follows easily from the modularity property
that $a(F,T)$ depends only on the equivalence class of $T$ modulo
$\SL(2,\Z)$. In fact, when $k$ is even, $a(F,T)$ depends only on the
equivalence class of $T$ modulo $\GL(2,\Z)$.
\par
We now let
$$
\weight^F_{k,d,\Lambda} = c_{k,d} \cdot d_\Lambda \cdot \frac{|a(d,
  \Lambda;F)|^2} { \langle F, F \rangle}
$$ 
where we put 
$$ 
d_\Lambda = 
\begin{cases} 
  1 & \text{ if } \Lambda^2 =1,\\
  2 & \text{ otherwise. }
\end{cases}
$$ 
and
\begin{equation}\label{eq:alf}
a(d, \Lambda;F)=\sum_{c\in \Cl_d}\overline{\Lambda(c)}a(F,c),
\end{equation}
a quantity which is well-defined in view of the invariance of Fourier
coefficients under $\SL(2,\Z)$. 
\par
\begin{remark}
  We will often consider $(d,\Lambda)$ to be fixed, and simplify the
  notation by writing $\weight_k^F$ only. Note that if $d=4$ and
  $\Lambda=1$, the weight $\weight_k^F$ is the same as the one defined
  in Theorem~\ref{th-average-l-functions}.
\end{remark}
\par
Next, we define the local spectral measures associated to the family
$\cusp_k^*$; we will show that they become equidistributed as
$k\rightarrow +\infty$ over even integers.  Let $\sop$ be a finite set
of primes. We have the components $\pi_{\sop}(F)=(\pi_p(F))_{p\in
  \sop}\in X_{\sop}$, and we define the measure $\nu_{\sop,k}$ on
$X_{\sop}$ by
$$
d\nu_{\sop,k} = d\nu_{\sop,k,d,\Lambda}=
\sum_{F \in \cusp_k^*} {\weight^F_{k,d,\Lambda} \delta_{\pi_{\sop}(F)}}
$$
where $\delta_{\bullet}$ is the Dirac mass at the given point. As we
will see, the normalization used has the effect that $\nu_{\sop,k}$ is
asymptotically a probability measure on $X_{\sop}$: we will show that
$$
\lim_{k\rightarrow +\infty} \nu_{\sop,2k}(X_{\sop})=1.
$$
\par
The local equidistribution problem --- which can clearly be phrased for
very general families of cusp forms --- is to determine if these
measures have limits as $k\rightarrow +\infty$, to identify  their
limit, to see in particular if the limit for a given $\sop$ is the
product of the limits for the subsets $\{p\}$, $p\in \sop$ (corresponding
to independence of the restrictions), and finally --- if possible --- to
express the resulting equidistribution in quantitative terms.
\par
To state our theorem, we now define the limiting measures. First of
all, we define a generalized Sato-Tate measure $\mu$ on each $X_p$ by
first taking the probability Haar measure on the space of conjugacy
classes of the compact unitary symplectic group $\USp(4)$, then
pushing this to a probability measure on $X$ by means of the map
$$
\begin{pmatrix}
e^{i \theta_1}&&&\\&e^{i \theta_2}&&\\&&e^{-i \theta_1}&\\&&&e^{-i \theta_2} 
\end{pmatrix}
\mapsto (e^{i\theta_1},e^{i\theta_2})\in X,
\quad\quad
(\theta_1,\theta_2)\in [0,\pi]^2, \quad \theta_1\le \theta_2,
$$
and finally extending it to $X_p$ by defining it equal to 0 outside $X$.
\par
In terms of the coordinates $(\theta_1, \theta_2)$ on $X$, the
resulting measure $\mu$ is explicitly given by
\begin{equation}\label{eq-haar-measure}
d\mu(\theta_1,\theta_2)=\frac{4}{\pi^2}(\cos\theta_1-\cos\theta_2)^2
\sin^2\theta_1\sin^2\theta_2 d\theta_1d\theta_2
\end{equation}
from the Weyl integration formula~\cite{fulhar}. We can also interpret
this measure as coming in the same way from conjugacy classes of the
unitary spin group $\USpin(5,\C)$, because of the ``exceptional
isomorphism'' $\Sp(4)\simeq \Spin(5)$.
\par
For each finite set of primes $\sop$, and $d, \Lambda$ as above, we
now define the measure $\mu_{\sop}=\mu_{\sop,d, \Lambda}$ on
$X_{\sop}$ by the formula
$$
d\mu_{\sop}=\bigotimes_{p\in \sop}d\mu_{p,d,\Lambda}
$$
where, for a single prime, we have
$$
d\mu_{p, d, \Lambda}=\Bigl(1-\Bigl(\frac{-d}{p}\Bigr)\frac{1}{p}
\Bigr)\Delta_p^{-1} d\mu
$$
(recall that $\mu$ is defined on $X_p$, but has support on $X$ only;
the same is therefore true of $\mu_{\sop}$) and the density functions
$\Delta_p=\Delta_{p,d,\Lambda}$ are given by
\begin{equation}\label{eq:density}
\Delta_p(\theta_1,\theta_2) =
\begin{cases}
  \big((1 + \frac{1}{p})^2 - \frac{4\cos^2\theta_1}{p}\big) \big((1 +
  \frac{1}{p})^2 - \frac{4\cos^2\theta_2}{p}\big) &
  \text{ if } p \text{ inert},\\
  \big((1 - \frac{1}{p})^2 + \frac{1}{p}(2\cos\theta_1 \sqrt{p} -
  \lambda_p) (\frac{2\cos\theta_1} {\sqrt{p}} - \lambda_p)\big)
  \\\quad\quad\times \big((1 - \frac{1}{p})^2 +
  \frac{1}{p}(2\cos\theta_2 \sqrt{p} - \lambda_p)
  (\frac{2\cos\theta_2} {\sqrt{p}} - \lambda_p)\big) & \text{ if } p
  \text{ split},
  \\
  \big(1 - \frac{2\lambda_p\cos\theta_1}{\sqrt{p}}+ \frac{1}{p}\big)
  \big(1 - \frac{2\lambda_p\cos\theta_2}{\sqrt{p}}+ \frac{1}{p}\big) &
  \text{ if } p \text{ ramified,}
\end{cases}
\end{equation}
where the behavior of primes refers to the field
$\Q(\sqrt{-d})$, and in the second and third cases, we put
$$
\lambda_p=\sum_{N(\p)=p}{\Lambda(\p)},
$$ 
the sum over the (one or two) prime ideals of norm $p$ in
$\Q(\sqrt{-d})$.
\par
Although we have written down this concrete, but unenlightening,
expression, there is a more intrinsic definition of these measures
$\mu_p = \mu_{p, d, \Lambda}$ and this will in fact be the way they
will naturally occur (and the way we will prove the results): they are
precisely what Furusawa and Shalika~\cite{furshal} call the
\emph{Plancherel measure for the local Bessel model associated to the
  data $(d,\Lambda)$}. In particular, this description shows that they
are probability measures, which is not quite obvious from the
definition. On the other hand, the following property, which is of
great relevance to global applications, is immediate: as $p\rightarrow
+\infty$, the measures $\mu_p$ converge weakly to the measure $\mu$,
which has a group-theoretic interpretation.
\par
Our local equidistribution result can now be stated:

\begin{theorem}[Local equidistribution and
  independence]\label{maintheorem}
  Fix any $d, \Lambda$ as above. For any finite set of primes $\sop$,
  the measures $\nu_{\sop,k}$ on $X_{\sop}$ converge weak-$*$ to
  $\mu_{\sop}$ as $k\rightarrow +\infty$ over even integers, i.e., for
  any continuous function $\varphi$ on $Y_{\sop}$, we have
$$
\lim_{k\ra +\infty} \sum_{F\in \cusp_{2k}^*}{\weight_{2k}^F
  \varphi((a_p(F),b_p(F))_{p\in \sop})}=
\int_{Y_{\sop}}{\varphi(x)d\mu_{\sop}(x)}.
$$
\par
In particular, if
$$
\varphi((y_p)_{p\in \sop})=\prod_{p\in \sop}{\varphi_p(y_p)}
$$
is a product function, we have
$$
\lim_{k\ra +\infty}
\sum_{F\in \cusp_{2k}^*}{\weight_{2k}^F \varphi((a_p(F),b_p(F))_{p\in \sop})}=
\prod_{p\in \sop}{\int_{Y_p}{\varphi_p(x)d\mu_p(x)}}.
$$
\par
Moreover, assume $\varphi$ is of product form, and that $\varphi_p$ is
a Laurent polynomial in $(a,b,a^{-1},b^{-1})$, invariant under the
action of the group $W$ given by~(\ref{eq-weyl-group}), and of total
degree $d_p$ as a polynomial in $(a+a^{-1}, b+b^{-1})$,  then we have
\begin{equation}\label{eq:quant-local}
\sum_{F\in \cusp_{2k}^*}{\weight_{2k}^F 
\varphi((a_p(F),b_p(F))_{p\in \sop})}=
\int_{Y_{\sop}}{\varphi(x)d\mu_{\sop}(x)}
+O\Bigl(
k^{-2/3}L^{1+\eps}\|\varphi\|_{\infty}
\Bigr)
\end{equation}
for any $\eps>0$, where
$$
L=\prod_{p\in \sop}{p^{d_p}},
$$
and $\|\varphi\|_{\infty}$ is the maximum of $|\varphi|$ on
$X^{|\sop|}\subset Y_{\sop}$. The implied constant depends only on $d$ and
$\eps$.
\end{theorem}

\begin{remark}
  It is possible to extend our results to odd $k$. However, this
  requires a slightly different definition of the weights
  $\weight_k^F$. For simplicity, we only consider $k$ even in this
  paper.
\end{remark}

\begin{remark}
  In a recent preprint, Sug Woo Shin~\cite{shin-plancherel} has proved
  a related result. In Shin's work, the weights $\weight_k^F$ are not
  present; instead the cusp forms are counted with the natural weight
  $1$. Using the trace formula, he proves a qualitative result that
  for suitable families of cusp forms on connected reductive groups
  over totally real fields, there is local equidistribution at a given
  place; when the level grows the sum of the point measures associated
  to the forms of fixed level converges towards the \emph{Plancherel}
  measure on the unitary dual of the local group.
\par
In fact, from the viewpoint of automorphic representations on
reductive groups, our result is essentially a (quantitative)
\emph{relative trace formula} analogue of what Shin (and others before
him, such as DeGeorge-Wallach~\cite{degeorge-wallach},
Clozel~\cite{clozel-limit}, Savin~\cite{savin-limit},
Serre~\cite{serre}, Sauvageot~\cite{sauvageot}) did using the trace
formula.
\par
We expect that the methods of this paper would suffice to prove the
equidistribution result for Siegel modular forms of level $N$ coprime
to the set of places $\sop$ as $N + k \rightarrow \infty$. It would
also be interesting to see if our results can be generalized to the
case of automorphic forms on the split special orthogonal groups,
using the formulas for the Bessel model there from~\cite{bff} (at
least qualitatively). We hope to treat these questions elsewhere.
\end{remark}

We briefly explain the structure of this paper. In
Chapter~\ref{bessel}, we introduce the Bessel model, explain its
relation to the Fourier coefficients and derive a result relating the
Fourier coefficients to Satake parameters. In Chapter~\ref{poincare}
we recall the definition of Poincar\'e series in this context and
derive a Petersson-type quantitative orthogonality formula for the
Siegel cusp forms (this involves non-trivial adaptations of the method
of Kitaoka~\cite{Kit}). In Chapter~\ref{equidistribution}, we put the
above results together to deduce our main theorem
(Theorem~\ref{maintheorem}) on local equidistribution. Finally, in
Chapter~\ref{applications}, we prove
Theorems~\ref{th-average-l-functions} and~\ref{th:low-lying} as well
as provide several other applications of the results of the previous
chapters.
\par
\medskip 
\par
\textbf{Acknowledgements.}  We would like to thank Masaaki Furusawa
for forwarding us a copy of the relevant part of his ongoing
work~\cite{furmarsha}, and for some helpful suggestions. Thanks also
to a referee for prompting us to give a more thorough explanation of
the behavior of Theorem~\ref{th:low-lying}.

\subsection{Notation}

We introduce here some notation used in the paper.

\begin{itemize}
\item The symbols $\Z$, $\Z_{\ge 0}$, $\Q$, $\R$, $\C$, $\Z_p$ and
  $\Q_p$ have the usual meanings. $\A$ denotes the ring of ad\`eles of
  $\Q$. For a complex number $z$, $e(z)$ denotes $e^{2\pi i z}$.
\item For any commutative ring $R$ and positive integer $n$, $M(n,R)$
  denotes the ring of $n$ by $n$ matrices with entries in $R$ and
  $\GL(n,R)$ denotes the group of invertible matrices in $M(n,R)$.  If
  $A\in M(n,R)$, we let $\T{A}$ denote its transpose.  We use
  $R^{\times}$ to denote $\GL(1,R)$.
\item For matrices $A$ and $B$, we use $A[B]$ to denote $^tBA B$,
  whenever the matrices are of compatible sizes. 
\item We say that a symmetric matrix in $M(n,\Z)$ is semi-integral if
  it has integral diagonal entries and half-integral off-diagonal
  ones.
\item Denote by $J_n$ the $2n$ by $2n$ matrix given by
$$
J_n =
\begin{pmatrix}
0 & I_n\\
-I_n & 0\\
\end{pmatrix}.
$$
We use $J$ to denote $J_2$.
\item For a positive integer $n$, define the algebraic group
  $\GSp(2n)$ over $\Z$ by
$$
\GSp(2n,R) = \{g \in \GL(2n,R) | \T{g}J_ng =
  \mu_n(g)J_n, \mu_n(g) \in R^\times\}
$$ 
for any commutative ring $R$.

Define $\Sp(2n)$ to be the subgroup of $\GSp(2n)$ consisting of
elements $g_1\in \GSp(2n)$ with $\mu_n(g_1)=1$.

The letter $G$ will always stand for $\GSp(4)$. The letter $\Gamma$
will always stand for the group $\Sp(4,\Z)$.

\item The Siegel upper-half space is defined by
$$
\H_n = \{ Z \in M_n(\C) | Z =\T{Z},\ \Imag(Z)
  \text{ is positive definite}\}.
$$ 
For 
$$
g=\begin{pmatrix} A&B\\ C&D \end{pmatrix} \in G(\R),
$$
and $Z\in \H_2$, we denote 
$$
J(g,Z) = CZ + D.
$$

\item For a prime $p$, the maximal compact subgroup $K_p$ of $G(\Q_p)$
  is defined by 
$$
K_p = G(\Q_p) \cap \GL(4,\Z_p).
$$

\item For a quadratic extension $\field$ of $\Q$ and $p$ a prime,
  define $\field_p =\field\otimes_\Q \Q_p$; we let $\Z_{\field}$
  denote the ring of integers of $\field$ and $\Z_{\field,p}$ its
  $p$-closure in $\field_p$.
\end{itemize}

\section{Bessel models}\label{bessel}

\subsection{Global Bessel models}

We recall the definition of the Bessel model of Novodvorsky and
Piatetski-Shapiro \cite{nov} following the exposition of Furusawa
\cite{fur}.
\par
Let $S \in M_2(\Q)$ be a symmetric matrix.\footnote{\ The notation
  conflicts a bit with the sets of primes $\sop$, but we hope that the
  boldface font of the latter and the context will prevent any
  confusion.} Define $\disc(S)=-4 \det(S)$ and put $d=-\disc(S)$. For
$$
S =\begin{pmatrix}
  a & b/2\\
  b/2 & c\\
\end{pmatrix},
$$ 
we define the element 
$$
\xi = \xi_S = \begin{pmatrix}
b/2 & c\\
-a & -b/2\\
\end{pmatrix}.
$$

Let $\field$ denote the subfield $\Q(\sqrt{-d})$ of $\C$. We always
identify $\Q(\xi)$ with $\field$ via
\begin{equation}\label{e:L} 
\Q(\xi)\ni x + y\xi \mapsto x +
y\frac{\sqrt{-d}}{2} \in \field, \ x,y\in \Q. 
\end{equation}

We define a $\Q$-algebraic subgroup $T =T_S$ of $\GL(2)$ by
\begin{equation}
T = \{g \in \GL(2) | \T{g}Sg =\det(g)
S\},
\end{equation} 
so that it is not hard to verify that $T(\Q) = \Q(\xi)^\times$. We
identify $T(\Q)$ with $\field^\times$ via~\eqref{e:L}.

We can also consider $T$ as a subgroup of $G$ via
\begin{equation}\label{gl2embed}
T \ni g \mapsto
\begin{pmatrix}
g & 0\\
0 & \det(g)\ \T{g^{-1}}
\end{pmatrix} \in G.
\end{equation}

Let us further denote by $U$ the subgroup of $G$ defined by
$$
U = \{u(X) =
\begin{pmatrix}
1_2 & X\\
0 & 1_2\\
\end{pmatrix} |\T{X} = X\},
$$
and by $R$ be subgroup $R=TU$ of $G$.

Let $\psi = \prod_v\psi_v$ be a character of $\A$ such that
\begin{itemize}
\item The conductor of $\psi_p$ is $\Z_p$ for all (finite) primes $p$,
\item $\psi_\infty(x) = e(x),$ for $x \in \R$,
\item $\psi|_\Q =1.$
\end{itemize} We define the
character $\theta = \theta_S$ on $U(\A)$ by
$$
\theta(u(X))=
\psi(\Tr(S(X))).
$$
\par
Let $\Lambda$ be a character of $T(\A) / T(\Q)$ such that $\Lambda |
\A^\times= 1$. Via~\eqref{e:L} we can think of $\Lambda$ as a
character of $\field^\times(\A)/\field^\times$ such that $\Lambda |
\A^\times = 1$. Denote by $\Lambda \otimes \theta$ the character of
$R(\A)$ defined by $(\Lambda \otimes \theta)(tu) =
\Lambda(t)\theta(u)$ for $t\in T(\A)$ and $u\in U(\A).$

Let $\pi$ be an automorphic cuspidal representation of $G(\A)$ with
trivial central character and $V_\pi$ be its space of automorphic
forms. For $\Phi \in V_\pi$, we define a function $B_\Phi$ on $G(\A)$
by
\begin{equation}\label{defbessel}
  B_\Phi(h) =
  \int_{R(\Q)Z_G(\A)\bs R(\A)} \overline{(\Lambda \otimes \theta)(r)}
\Phi(rh)dr.
\end{equation}
\par
The $\C$-vector space of functions on $G(\A)$ spanned by $\{B_\Phi |
\Phi \in V_\pi \}$ is called the \emph{global Bessel space} of type
$(S, \Lambda, \psi)$ for $\pi$; it is invariant under the regular
action of $G(\A)$ and when the space is non-zero, the corresponding
representation is a model of $\pi$. Thus one says that $\pi$ has a
global Bessel model of type $(S, \Lambda, \psi)$ if this global Bessel
space is non-zero, i.e., if there exists $\Phi \in V_\pi$ such that
$B_\Phi \neq 0$.

\subsection{The classical interpretation}\label{s:classical} 

Let us now suppose that $\Phi$, $\pi$ come from a classical Siegel
cusp form $F$. More precisely, for a positive integer $N$,
define 
$$
\Gamma^\ast(N) := \{g \in \Gamma: g
\equiv \begin{pmatrix}\ast&0&0&0\\ 0&\ast&0&0\\0&0&\ast&0\\
  0&0&0&\ast \end{pmatrix} \pmod{N} \}.
$$ 
\par
We say that $F \in \cusp_k(\Gamma^\ast(N))$ if it is a holomorphic
function on $\H_2$ which satisfies
$$
F(\gamma Z) = \det(J(\gamma,Z))^k F(Z)
$$ 
for $\gamma \in \Gamma^\ast(N)$, $Z \in \H_2$, and vanishes at the
cusps. It is well-known that $F$ has a Fourier expansion $$F(Z)
=\sum_{T > 0, T\in \L} a(F, T) e(\Tr(TZ)),$$ where $e(z) = \exp(2\pi
iz)$ and $T$ runs through all symmetric positive-definite matrices of
size two in a suitable lattice $\L$ (depending on $N$). If $N=1$, then
$\L$ is just the set of symmetric, semi-integral matrices.  Also,
recall that $\cusp_k(\Gamma^\ast(1))$ is denoted simply by $\cusp_k$.

We define the ad\'elization $\Phi_F$ of $F$ to be the function on
$G(\A)$ defined by
\begin{equation}\label{adelization}
\Phi_F(\gamma h_\infty k_0) =
  \mu_2(h_\infty)^k\det(J(h_\infty,
  iI_2))^{-k}F(h_\infty(i))
\end{equation} 
where $\gamma \in G(\Q), h_\infty \in G(\R)^+$ and 
$$
k_0 \in \prod_{ p \nmid N } K_p \prod_{p |N} K_p^N 
$$
where 
$$
K_p^N = \{g \in G(\Z_p):g
\equiv \begin{pmatrix}\ast&0&0&0\\ 0&\ast&0&0\\0&0&\ast&0\\
  0&0&0&\ast \end{pmatrix} \pmod{N} \}
$$ 
is the local analogue of $\Gamma^\ast(N)$. Then $\Phi_F$ is a
well-defined function on the whole of $G(\A)$ by strong approximation,
and is an automorphic form.

We now assume $N=1$. Let $d$ be a positive integer such that $-d$ is
the discriminant of the imaginary quadratic field
$\Q(\sqrt{-d})$. Define
\begin{equation}\label{defsd}
 S = S(-d) = \begin{cases} \begin{pmatrix}
  \frac{d}{4} & 0\\
 0 & 1\\\end{pmatrix} & \text{ if } d\equiv 0\pmod{4}, \\[4ex]
 \begin{pmatrix} \frac{1+d}{4} & \frac12\\\frac12 & 1\\
 \end{pmatrix} & \text{ if } d\equiv 3\pmod{4}.\end{cases}
\end{equation}
\par
Define the groups $R, T, U$ as in the previous section.

Put $\field=\Q(\sqrt{-d})$. Recall that $\Cl_d$ denotes the ideal
class group of $\field$. Let $(t_c)$, $c\in \Cl_d$, be coset
representatives such that
\begin{equation}\label{e:tcosetca1}
T(\A) = \coprod_{c}t_cT(\Q)T(\R)\Pi_{p<\infty}
(T(\Q_p) \cap \GL_2(\Z_p)),
\end{equation}
with $t_c \in \prod_{p<\infty} T(\Q_p)$. We can write
$$
t_c = \gamma_{c}m_{c}\kappa_{c}
$$
with $\gamma_{c} \in \GL(2,\Q)$, $m_{c} \in \GL^+(2,\R)$, and
$\kappa_{c}\in \Pi_{p<\infty} \GL(2,\Z_p).$

The matrices 
$$
S_{c} = \det(\gamma_{c})^{-1}\ (\T{\gamma_{c}})S\gamma_{c}
$$
have discriminant $-d$, and form a set of representatives of the
$\SL(2,\Z)$-equivalence classes of primitive semi-integral positive
definite matrices of discriminant $-d$.

Choose $\Lambda$ a character of $T(\A)/T(\Q)T(\R)((\Pi_{p<\infty}
T(\Z_p))$, which we identify with an ideal class character of
$\Q(\sqrt{-d})$.

Next, for any positive integer $N$, define (this is a certain ray
class group)
$$
\Cl_d(N) = T(\A) /T(\Q)T(\R)\Pi_{p<\infty} (T(\Q_p) \cap
K_p^{(0)}(N)),
$$ 
where $K_p^{(0)}(N)$ is the subgroup of $\GL_2(\Z_p)$
consisting of elements that are congruent to a matrix of the form
$$
\begin{pmatrix}*&0\\ *&*\end{pmatrix} \pmod{N}.
$$

As before, we can take coset representatives $(t_{c'})$, $c'\in
\Cl_d(N)$, such that
\begin{equation}\label{e:tcosetca12}
T(\A) = \coprod_{c'}t_{c'}T(\Q)T(\R)\Pi_{p<\infty}
(T(\Q_p) \cap K_p^{(0)}(N)),
\end{equation}
with $t_{c'} \in \prod_{p<\infty} T(\Q_p),$ and write
$$
t_{c'} = \gamma_{c'}m_{c'}\kappa_{c'}
$$
with $\gamma_{c'} \in \GL(2,\Q)$, $m_{c'} \in \GL^+(2,\R)$, and
$\kappa_{c'}\in \Pi_{p<\infty} K_p^{(0)}(N).$
\par
Define the matrices
$$
S_{c'} = \det(\gamma_{c'})^{-1}\ (\T{\gamma_{c'}})S\gamma_{c'}.
$$
\par
For each pair of positive integers $L,M$, define the element $H(L,M)
\in G(\A)$ by 
$$
H(L,M)_\infty = 1,\quad\quad
H(L,M)_p
= \begin{pmatrix}
LM^2&&&\\
&LM&&\\
&&1&\\&&&
M
\end{pmatrix}
$$ 
for each prime $p < \infty$. Note that $H(1,1) = 1$.

For any symmetric matrix $T$, we let $T^{L,M}$ denote the matrix
\begin{equation}\label{eq-sclm}
T^{L,M}
=\begin{pmatrix}
L&\\&L\end{pmatrix}
\begin{pmatrix}M&\\&1\end{pmatrix}
T
\begin{pmatrix}M&\\&1\end{pmatrix}.
\end{equation}
\par
Define the quantity $B(L, M; \Phi_F)$
by
\begin{equation}\label{blmdef}
  B(L, M; \Phi_F) = B_{\Phi}(H(L,M))=
  \int_{R(\Q)Z_G(\A)\bs R(\A)} \overline{(\Lambda \otimes\theta)(r)}
  \Phi_F(rH(L,M))dr.
\end{equation}

The next proposition proves an important relation which expresses the
quantity $B(L, M; \Phi_F)$ in terms of Fourier coefficients. The proof
is fairly routine, but to our best knowledge it has not appeared in
print before.

\begin{proposition}\label{classicalbesselprop} 
  Let $F \in \cusp_k$ have the Fourier expansion
$$
F(Z)=\sum_{T > 0} a(F, T) 
e(\Tr(TZ)).
$$ 
\par
Then we have
$$
B(L, M; \Phi_F) = r \cdot e^{-2 \pi \Tr (S)} (LM)^{-k} \frac{1}{|
  \Cl_d(M) |}\sum_{c\in \Cl_d(M)} \overline{\Lambda(c)}a(F,S_c^{L,M}).
$$ 
where $r$ is a non-zero constant depending only on the normalization
used for the Haar measure on $R$.
\end{proposition}

\begin{proof}
  For the purpose of this proof, we shorten $H(L,M)$ to $H$ whenever
  convenient.  Note that 
$$B(L, M; \Phi_F) = B(1, 1; \Phi_F^{L,M})
$$
where the automorphic form $\Phi_F^{L,M}$ is given by 
$$
\Phi_F^{L,M} (g) = \Phi_F(gH).
$$

Define 
$$
H_\infty = \begin{pmatrix}LM^2&&&\\&LM&&\\&&1&\\&&&M\end{pmatrix} \in
G(\R)^+,
$$
and define the Siegel modular form 
\begin{equation}\label{deffpr}
  F'(Z) =  (LM)^{-k}F(H_\infty^{-1}Z),
\end{equation}
which is in $\cusp_k(\Gamma^\ast(LM^2))$, as one can easily show.
\par
Let $\Phi_{F'}$ be the ad\'elization of $F'$ as defined
by~\eqref{adelization}. We claim that $\Phi_F^{L,M} = \Phi_{F'}$. To
see this, since both functions are right invariant under the group
$$
\prod_{ p \nmid  LM^2 } K_p \prod_{p |LM^2} K_p^{LM^2},
$$ 
it is enough to show that $\Phi_F^{L,M}(g_\infty) =
\Phi_{F'}(g_\infty)$ for $g_\infty \in G(\R)^+$. This is shown by the
following computation:
\begin{align*} 
  \Phi_F^{L,M}(g_\infty) &= \Phi_F(g_\infty H)\\
  &= \Phi_F(H_\infty^{-1}g_\infty) \\
&=  \mu_2(H_\infty^{-1}g_\infty)^k\det(J(H_\infty^{-1}g_\infty,
  iI_2))^{-k} F(H_\infty^{-1}g_\infty(i))\\
&=  (LM)^{-k}\mu_2(g_\infty)^k\det(J(g_\infty, iI_2))^{-k}
  F(H_\infty^{-1}g_\infty(i)) \\&=\Phi_{F'}(g_\infty).
\end{align*}
\par
Hence we are left with the problem of evaluating $B(1, 1;
\Phi_{F'})$. Note that $\Phi_{F'}$ is right invariant under
$K_p^{(0)}(M)$ (where we think of $\GL_2$  as a subgroup of $\GSp_4$ via~\eqref{gl2embed}).  Using~\eqref{e:tcosetca12}, the same arguments as
in~\cite[Prop. 2.8.5]{saha-thesis}, give us
$$
B(1, 1; \Phi_{F'}) = e^{-2 \pi \Tr (S)}
\frac{1}{| \Cl_d(M) |} \sum_{c\in \Cl_d(M)} \overline{\Lambda(c)}a(F',S_c)
$$ 
for a suitably normalized Haar measure, where $a(F',T)$ denotes the
Fourier coefficients of $F'$. Using~\eqref{deffpr}, one can easily
check that 
$$
a(F',S_c) = (LM)^{-k}a(F,S_c^{L,M}),
$$
and this completes the proof.
\end{proof}

\begin{remark}
  In an earlier preprint version of this paper, we had claimed such a
  result with a sum over $c$ in $\Cl_d$ instead of $\Cl_d(M)$. This
  was incorrect (when $M\not=1$); the mistake in the proof was to assume that
   $\Phi_{F'}$ is invariant under the bigger subgroup $\GL_2(\Z_p)$ when
  arguing as in~\cite[Prop. 2.8.5]{saha-thesis}.
\end{remark}

\begin{remark} 
  The above result is one of the three crucial ingredients that are
  required for the proof of the asymptotic Petersson-type formula
  (Proposition~\ref{pr:petersson}) which forms the technical heart of
  this paper. The other two ingredients are Sugano's formula
  (Theorem~\ref{t:sugano}) and the asymptotic orthogonality for
  Poincare series (Proposition~\ref{proppetersson2}).
\end{remark}

\subsection{Local Bessel models and Sugano's
  formula}\label{s:localbesselsugano} 

Let $\pi = \otimes_v \pi_v$ be an irreducible automorphic cuspidal
representation of $G(\A)$ with trivial central character and $V_\pi$
be its space of automorphic forms. We assume that $\pi$ is unramified
at all finite places. Let $S$ be a positive definite, symmetric,
semi-integral matrix such that $-d = -4 \det(S)$ is the discriminant
of the imaginary quadratic field $\field=\Q(\sqrt{-d})$. Let $\psi$,
$\Lambda$ be defined as in the previous Section. Define the groups $R,
T, U$ as before and the Bessel function $B_\Phi$ on $G(\A)$ as
in~\eqref{defbessel}, for a function $\Phi = \prod_v \Phi_v$ which is
a pure tensor in $\pi$.

For a finite prime $p$, we use $\left(\frac{\field}{p}\right)$ to
denote the Legendre symbol; thus $\left(\frac{\field}{p}\right)$
equals $-1$, $0$ or $1$ depending on whether the prime is inert,
ramified or split in $\field$. In the latter two cases, we use
$p_{\field}$ to denote an element of $\field_p = \field \otimes_\Q
\Q_p$ such that $N_{\field/\Q} (p_{\field}) \in p \Z_p^\times.$

Outside $v=\infty$, the local representations are unramified spherical
principal series. Therefore, by the uniqueness of the Bessel model for
$G$, due to Novodvorsky and Piatetski-Shapiro~\cite{nov}, we have
\begin{equation}\label{besseluniqueness} 
  B_\Phi(g) =  B_\Phi(g_{\infty}) \prod_{p} B_p(g_p)
\end{equation}
where $B_p$ is a local Bessel function on $G(\Q_p)$, the definition of
which we will now recall.

\begin{remark}
  If the global Bessel space is zero, then both sides
  of~\eqref{besseluniqueness} are zero. In
  particular,~\eqref{besseluniqueness} remains valid regardless of
  whether our choice of $S$ and $\Lambda$ ensures a non-zero Bessel
  model.
\end{remark}

To describe the local Bessel function $B_p$ for a prime $p$, let $\B$
be the space of locally constant functions $\varphi$ on $G(\Q_p)$
satisfying
$$
\varphi(tuh)= \Lambda_p(t)\theta_p(u)\varphi(h), \text{ for } t\in
T(\Q_p),u \in U(\Q_p), h \in G(\Q_p).
$$ 
\par
Then by Novodvorsky and Piatetski-Shapiro~\cite{nov}, there exists a
unique subspace $\B(\pi_p)$ of $\B$ such that the right regular
representation of $G(\Q_p)$ on $\B(\pi_p)$ is isomorphic to
$\pi_p$. Let $B_p$ be the unique $K_p$-fixed vector in $\B(\pi_p)$
such that $B_p(1) =1$. Therefore we have
\begin{equation}\label{e:bformula}
  B_p(tuhk)= \Lambda_p(t)\theta_p(u)B_p(h),
\end{equation} 
for $t\in T(\Q_p),u \in U(\Q_p), h \in G(\Q_p), k\in K_p$.
\par
Let $h_p(l,m) \in G(\Q_p)$ be the matrix defined as follows:
$$
h_p(l,m):=\begin{pmatrix}p^{l+2m}&&&\\&p^{l+m}&&\\&&1&\\&&&p^{m}
\end{pmatrix}.
$$ 
\par
As explained in~\cite{fur}, the local Bessel function $B_p$ is
completely determined by its values on $h_p(l,m).$ An explicit formula
for $B_p(h_p(l,m))$ in terms of the Satake parameters is stated
in~\cite{bff}. This formula can be neatly encapsulated in a generating
function, due to Sugano~\cite{sug}, which we now explain.
\par
Because $\pi_{p}$ is spherical, as recalled earlier, it is the
unramified constituent of a representation $\chi_1 \times \chi_2
\rtimes \sigma$ induced from a character of the Borel subgroup
associated to unramified characters $\chi_1, \chi_2, \sigma$ of
$\Q_p^\times$, and because it has trivial central character (since
$\pi$ does) we have $\chi_1 \chi_2 \sigma^2 =1$. Let us put
$(a_p,b_p)=(\sigma(p),\sigma(p)\chi_1(p))\in Y_p$, and as in the
definition of the measure $\mu_p$, let
$$
\lambda_p=\sum_{\substack{x \in \field_p^\times / \Z_{\field,
      p}^\times\\N(x)=p}}{\Lambda_p(x)},
$$
where the number of terms in the sum is
$1+\left(\frac{\field}{p}\right)$.
\par
The next Theorem is due to Sugano~\cite[p. 544]{sug} (the reader may
also consult~\cite[(3.6)]{fur}).

\begin{theorem}[Sugano]\label{t:sugano}
  Let $\pi$ be an unramified spherical principal series representation
  of $G(\Q_p)$ with associated local parameters $(a,b)\in Y_p$ and
  spherical Bessel function $B_p$ as above. Then we have
\begin{equation}\label{e:sugano}
B_p(h_p(l,m))= p^{-2m-\frac{3l}{2}} U^{l, m}_p(a,b) 
\end{equation}
where for each $l,m\geq 0$, the function 
$$
U^{l, m}_p(a,b) = U^{l,m}_p(a,b;\field_p,\Lambda_p)
$$ 
is a Laurent polynomial in $\C[a,b,a^{-1}, b^{-1}]$, invariant under
the action of the Weyl group~(\ref{eq-weyl-group}), which depends only
on $p$, $\left(\frac{\field}{p}\right)$ and $\lambda_p$.
\par
More precisely, the generating function 
\begin{equation}
C_p(X,Y)=
C_p(X,Y;a,b) = \sum_{l \ge 0}\sum_{m \ge 0} U^{l,
    m}_p(a,b)X^mY^l 
\end{equation} 
is a rational function given by
\begin{equation}\label{defcp}
C_p(X,Y) =  \frac{H_p(X,Y)}{P_p(X)Q_p(Y)} 
\end{equation} 
where
\begin{align*}
  P_p(X) &= (1-abX)(1-ab^{-1}X)(1-a^{-1}bX)(1-a^{-1}b^{-1}X),
  \\
  Q_p(Y) &= (1-aY)(1-bY)(1-a^{-1}Y)(1-b^{-1}Y),
  \\
  H_p(X,Y) &= (1 + XY^2)\big(M_1(X)(1+X)+ p^{-1/2}\lambda_p\sigma(a,b)
  X^2\big) \\
  & \quad\quad- XY\big(\sigma(a,b) M_1(X)-
  p^{-1/2}\lambda_pM_2(X)\big) - p^{-1/2}\lambda_pP_p(X)Y +
  p^{-1}\left(\frac{\field}{p}\right)P_p(X)Y^2,
\end{align*}
in terms of auxiliary polynomials given by
\begin{align*} 
  \sigma(a,b) &= a +b +a^{-1} + b^{-1}, \qquad 
  \tau(a,b) = 1+ab+ ab^{-1}+a^{-1}b+a^{-1}b^{-1},\\
  M_1(X) &= 1 - \left(p - \left(\frac{\field}{p}\right)\right)^{-1}
  \left(
    p^{1/2}\lambda_p\sigma(a,b)
    -\left(\frac{\field}{p}\right)(\tau(a,b)-1)-\lambda_p^2 
  \right)  X 
  - p^{-1}\left(\frac{\field}{p}\right)X^2,
  \\
  M_2(X) &= 1 - \tau(a,b) X - \tau(a,b) X^2 + X^3.
\end{align*}
\end{theorem}

\begin{remark}
For instance, we note that
$$
U^{0, 0}_p(a,b) = 1
$$ 
and that
\begin{equation}\label{eq-u10}
  U^{1, 0}_p(a,b) = \sigma(a,b)-p^{-1/2}\lambda_p
  =a +b +a^{-1}+ b^{-1} - p^{-1/2}\lambda_p.
\end{equation}
\par
We also note that taking $X=0$ leads to the simple formula
\begin{equation}\label{e:onevar-sugano}
  \sum_{l \ge 0} U^{l, 0}_p(a,b)Y^l =\frac{
    1 -p^{-\frac{1}{2}}\lambda_pY +
    p^{-1}\left(\frac{\field}{p}\right)Y^2}{Q_p(Y)},
\end{equation}
which we will use later on. The formula for $Y=0$ is more complicated,
but we note (also for further reference) that it implies the formula
\begin{equation}\label{eq-u01}
  U^{0,1}_p(a,b)=\tau(a,b) 
  - \left(p - \left(\frac{\field}{p}\right)\right)^{-1}
  \left(p^{1/2}\lambda_p\sigma(a,b)
    -\left(\frac{\field}{p}\right)(\tau(a,b)-1)-\lambda_p^2 
  \right).
\end{equation}
\end{remark}

As was the case for the definition of the measures $\mu_p$, we have
written down a concrete formula. These are not very enlightening by
themselves (though we will use the special cases above), and the
intrinsic point of view is that of \emph{Macdonald
  polynomials}~\cite{macd} associated to a root system. In particular,
this leads to the following important fact:

\begin{proposition}\label{pr:pol-basis}
  Let $(d,\Lambda)$ be as before. For any fixed prime $p$, the
  functions
$$
(a,b)\mapsto U_p^{l,m}(a,b)
$$
where $l$, $m$ run over non-negative integers, form a basis of the
space of Laurent polynomials in $\C[a,b,a^{-1},b^{-1}]$ which are
invariant under the group $W$ generated by the three transformations
above. 
\par
Moreover, any such Laurent polynomial $\varphi$ which has total degree
$d$ as polynomial in the variables $(a+a^{-1},b+b^{-1})$ can be
represented as a combination of polynomials $U_{p}^{l,m}(a,b)$ with
$l+2m\leq d$.
\end{proposition}

\begin{proof} 
  Because of Theorem~\ref{t:sugano}, we can work with the Bessel
  functions $B_{p}(h_{p}(l, m))$ instead. But as shown in detail
  in~\cite[\S 3]{furmarsha}, these unramified Bessel functions are
  (specializations of) Macdonald polynomials associated to the root
  system of $G$, in the sense of~\cite{macd}. By the theory of
  Macdonald polynomials, these unramified Bessel functions form a
  basis for the Laurent polynomials in two variables that are
  symmetric under the action of the Weyl group $W$.
\par
The last statement, concerning the $U_p^{l,m}$ occurring in the
decomposition of $\varphi$ of bidegree $(d,d)$, can be easily proved
by induction from the corresponding fact for the coefficients (say
$\tilde{U}_{l,m}(a,b)$) of the simpler generating series
$$
\frac{1}{P_p(X)Q_p(Y)}=\sum_{l,m\geq 0}{\tilde{U}_{l,m}(a,b)X^mY^l},
$$
for which the stated property is quite clear. (It is also a standard
fact about the characters of representations of $\USp(4,\C)$, since
$\sigma(a,b)$ and $\tau(a,b)$ are the characters of the two
fundamental representations acting on a maximal torus.)
\end{proof}




\begin{lemma}\label{lm:bounds-ulm}
  Let $(d,\Lambda)$ be as before. Let $\sop$ be a finite set of primes
  and $(l_p)$, $(m_p)$ be tuples of non-negative integers, indexed by
  $\sop$. There exists an absolute constant $C\geq 0$ such that for
  every $(x_p)_{p\in \sop}=(a_p,b_p)\in X^{\sop}$, i.e., parameters of
  tempered representations, we have
$$
\Bigl|\prod_{p\in \sop}{U_p^{l_p,m_p}(a_p,b_p)}
\Bigr|\leq C^{|\sop|}
\prod_{p\in \sop}{(l_p+3)^3(m_p+3)^3}.
$$
\end{lemma}

\begin{proof}
  It is enough to prove this when $\sop=\{p\}$ is a single prime, and
  $l_p=l$, $m_p=m\geq 0$. Then by Sugano's formula, the polynomial
  $U_p^{l,m}(a,b)$ is a linear combination of at most $14$ polynomials
  of the type arising in the expansion of the denominator only, i.e.,
  of
$$
\frac{1}{P_p(X)Q_p(Y)},
$$
and moreover the coefficients in this combination are absolutely
bounded as $p$ varies (they are either constants or involve quantities
like $p^{-1/2}$).
\par
Expanding in geometric series and using $|a|= 1$, $|b|=1$, the
coefficient of $X^mY^l$ in the expansion of the denominator is a
product of the coefficient of $X^m$ and that of $Y^l$; each of them is
a sum, with coefficient $+1$, of $\leq (m+3)^3$ (resp. $(l+3)^3$)
terms of size $\leq 1$. The result follows from this.
\end{proof}

\begin{remark}
  Sugano's formula explicitly computes the Bessel function in terms of
  Satake parameters in the case of an unramified representation. The
  other case where an explicit formula for the Bessel function at a
  finite place is known is when $\pi_p$ is Steinberg,
  see~\cite{lfshort},~\cite{pitale-bessel}.
\end{remark}

\subsection{The key relation}\label{s:keyrel}

We consider now Siegel modular forms again. Let
$$
F(Z) = \sum_{T > 0} a(F,T) e(\Tr(TZ)) \in \cusp_k
$$ 
be an eigenfunction for all the Hecke operators. Define its
ad\'elization $\Phi_{F}(g)$ by~\eqref{adelization}. This is a function
on $G(\Q)\bs G(\A)$ and we may consider the representation of $G(\A)$
generated by it under the right-regular action. Because we do not have
strong multiplicity one for $G$, we can only say that this
representation is a \emph{multiple} of an irreducible representation
$\pi_F$. However, the unicity of $\pi_F$, as an isomorphism class of
representations of $G(\A)$, is enough for our purposes.\footnote{\
  Added in proof: in a recent preprint, Narita, Pitale and Schmidt
  show that $\Phi_F$ does indeed generate an irreducible
  representation.}
\par
We can factor $\pi_F =\otimes \pi_v(F)$ where the local
representations $\pi_v$ are given by:
$$
\pi_v(F) =  
\begin{cases}\text{holomorphic discrete series} & \text{ if } v=\infty,\\
  \text{unramified spherical principal series} &\text{ if } v \text{
    is finite, }
\end{cases}
$$
and we denote by $(a_p(F),b_p(F))\in Y_p$, the local parameters
corresponding to the local representation $\pi_p(F)$ at a finite
place. 
\par
Let once more $d$ be a positive integer such that $-d$ is a
fundamental discriminant and define $S$ as in~\eqref{defsd}. Choose an
ideal class character $\Lambda$ of $\field$. Let the additive
character $\psi$, the groups $R, T, U$ and the matrices $S_c$,
$S_{c'}^{L,M}$ be defined as in Section~\ref{s:classical}.  For
positive integers $L, M$, define $B(L, M; \Phi_F)$
by~\eqref{blmdef}. Then, by the uniqueness of the Bessel model
(i.e.,~(\ref{besseluniqueness}), we have
$$
B(L, M; \Phi_F) = B(1, 1; \Phi_F)\prod_{p}
B_{p}(h_{p}(l_p,m_p)) =
B(1, 1; \Phi_F)\prod_{p\mid LM}
B_{p}(h_{p}(l_p,m_p)),
$$
where $l_p$ and $m_p$ are the $p$-adic valuations of $L$ and $M$
respectively.  Now, using Sugano's formula~(\ref{e:sugano}) and twice
Proposition~\ref{classicalbesselprop} -- which has the effect of
canceling the constant $r\not=0$ that appears in the latter --, we
deduce:

\begin{theorem}\label{t:relation} 
  Let $(d,\Lambda)$ be as before, let $p$ be prime and let $U^{l,
    m}_p(a,b)$ be the functions defined in Theorem~\ref{t:sugano}. For
  any $F\in\cusp_{2k}^*$ and integers $L$, $M\geq 1$, we have
$$
\frac{|\Cl_d|}{|\Cl_d(M)|}\sum_{c'\in \Cl_d(M)}
\overline{\Lambda(c')}a(F,S_{c'}^{L,M}) =
L^{k-\frac{3}{2}}M^{k-2}\sum_{c\in\Cl_d} \overline{\Lambda(c)}a(F,S_c)
\prod_{p\mid LM}{U^{l_p, m_p}_{p}(a_{p}(F),b_{p}(F))},
$$
where $l_p$ and $m_p$ are the $p$-adic valuations of $L$ and $M$
respectively. 
\end{theorem}

The point of this key result is that it allows us to study functions
of the Satake parameters of $\pi_F$ using Fourier coefficients of $F$,
although there is no direct identification of Hecke eigenvalues with
Fourier coefficients.

\begin{remark}
  This relation holds for every $\Lambda$, but we can not remove the
  sum over $c$ by Fourier inversion because the functions
  $U_p^{l_p,m_p}$ \emph{depend on $\Lambda$}.
\end{remark}

\section{Poincar\'e series, Petersson formula and
  orthogonality}\label{poincare}

The relation given by Theorem~\ref{t:relation} between Fourier
coefficients of $F$ on the one hand, and functions of the spectral
Satake parameters of $\pi_F$ on the other, will enable us to deduce
equidistribution results for Satake parameters from asymptotics for
Fourier coefficients. For this, we need a way to understand averages
of Fourier coefficients of Siegel forms $F$ in a suitable family; this
will be provided by a variant of the classical Petersson formula. In
order to prove the latter, we follow the standard approach: we
consider Poincar\'e series and study their Fourier coefficients.

\subsection{Poincar\'e series and the Petersson formula}

Given a symmetric semi-integral positive-definite matrix $Q$ of size
two, the $Q$-th Poincar\'e series of weight $k$, denoted $P_{k, Q}$,
is defined as follows:
$$
P_{k, Q}(Z) = \sum_{\gamma \in \Delta\bs \Gamma}\det(J(\gamma,
Z))^{-k}e(\Tr(Q \gamma (Z)))
$$ 
where $\Delta$ is the subgroup of $\Gamma$ consisting of matrices of
the form $\begin{pmatrix}1&U\\0&1\end{pmatrix}$, with $U$ symmetric.
\par
It is known that $P_{k,Q}$ is absolutely and locally uniformly
convergent for $k\geq 6$, and defines an element of $\cusp_k$ (as
first proved by Maass). In fact, any Siegel cusp form $F\in \cusp_k$
is a linear combination of various $P_{k, Q}$ (with $Q$ varying). This
follows from the basic property of Poincar\'e series: they represent,
in terms of the Petersson inner product, the linear forms on $\cusp_k$
given by Fourier coefficients. Precisely, for $F \in \cusp_k$ with
Fourier expansion
$$
F(Z) = \sum_{T > 0} a(F,T) e(\Tr(TZ)),
$$ 
we have the crucial identity
\begin{equation}\label{eq:poincare-basic}
\langle F, P_{k, T} \rangle = 8 c_k (\det T)^{-k+ 3/2}a(F,T),
\end{equation}
where
\begin{equation}\label{defck}
c_k =  \frac{\sqrt{\pi}}{4}(4 \pi)^{3 -2k}
  \Gamma(k - \tfrac{3}{2})\Gamma(k-2).
\end{equation}
(see~\cite{kohp} or~\cite[p. 90]{klingen} for instance).
\par
\par
We are interested in the limiting behavior of $a(k;c,c',L,M)$ as $k\ra
+\infty$. The following qualitative result was proved in~\cite{kst}:

\begin{proposition}[Asymptotic orthogonality, qualitative
  version]\label{proppetersson}
  For $L$, $M\geq 1$, $c\in \Cl_d$ and $c'\in \Cl_d(M)$, let
$$
a(k;c,c',L,M)=a(P_{k,S_{c}},S_{c'}^{L,M})
$$
denote the $S_{c'}^{L,M}$-th Fourier coefficient of the Poincar\'e series
$P_{k, S_{c}}$. Then,
we have
$$
a(k;c,c',L,M) \ra |\Aut(c)| \cdot \delta(c,c';L;M)
$$
as $k\ra + \infty$ over the even integers. Here
$$
\delta(c,c';L;M)= \begin{cases} 1 &\text{ if } L=1, M=1\text{ and $c$
    is $\GL(2,\Z)$-equivalent to $c'$,}\\
  0 & \text{ otherwise, }
\end{cases}
$$
and $|\Aut(c)|$ is the finite group of integral points in the
orthogonal group $O(T)$ of the quadratic form defined by $c$.
\end{proposition}

\begin{remark}
  Since this will be a subtle point later on, we emphasize that
  $\delta(c,c';L,M)=1$ when $c$ and $c'$ are invariant under
  $\GL(2,\Z)$, not under $\SL(2,\Z)$.
\end{remark}

This is sufficient for some basic applications, but (for example) to
handle the low-lying zeros, we require a quantitative version. We
prove the following:


\begin{proposition}[Asymptotic orthogonality, quantitative
  version]\label{proppetersson2}
With notation as above, we have
$$
a(k;c,c',L,M) =  |\Aut(c)| \cdot \delta(c,c';L;M) + L^{k-3/2}M^{k-2}
A(k;c,c',L,M)
$$ 
where
$$
A(k;c,c',L,M) \ll L^{1 + \eps}M^{3/2 + \eps}k^{-2/3}
$$
for any $\eps>0$, the implied constant depending
only on $\eps$ and $d$.
\end{proposition}


In the proof, for conciseness, we will write $|A|$ for the determinant
of a matrix. The basic framework of the argument is contained in the
work of Kitaoka~\cite{Kit}, who proved an estimate for the Fourier
coefficients $a(P_{k,Q},T)$ of Poincar\'e series for fixed $Q$ and
$k\geq 6$, in terms of the determinant $\det(T)$ (and deduced from
this an estimate for Fourier coefficients of arbitrary Siegel cusp
forms in $\cusp_k$, since the space is spanned by Poincar\'e
series). 
\par
However, Kitaoka considered $k$ to be fixed; our goal is to have a
uniform estimate in terms of $L$, $M$ and $k$, and this requires more
detailed arguments.
\par
In particular, we will require the following quite standard
asymptotics for Bessel functions:
\begin{align}\label{besselasymptotic1}
  J_{k}(x) &\ll \frac{x^k}{\Gamma(k+1)},\quad  \text{ if } k\geq 1,\ 0\leq
  x\ll \sqrt{k+1},\\
  J_{k}(x) &\ll \min(1,xk^{-1})k^{-1/3},\quad  \text{ if } k\geq 1,\ x\geq
  1,\label{besselasymptotic2}
\\
  J_{k}(x) &\ll \frac{2^k}{\sqrt{x}},\quad  \text{ if } k\geq 1,\ x>0,
\label{besselasymptotic3}
\end{align}
where the implied constants are absolute (the first inequality follows
from the Taylor expansion of $J_k(z)$ at $z=0$,
the second is~\cite[(2.11)]{ils}, and the third, which is very rough,
by combining $|J_k(x)|\leq 1$ when $x\leq 2k$, and,
e.g.,~\cite[(2.11')]{ils} for $x\geq 2k$).

\begin{proof}[Proof of Proposition~\ref{proppetersson2}]
Let
$$
T=S_{c'}^{L,M},\quad\quad Q=S_{c},
$$
so that we must consider the $T$-th Fourier coefficients of $P_{k,Q}$
(this notation, which clashes a bit with the earlier one for the torus
$T$, is chosen to be the same as that in~\cite{Kit}, in order to
facilitate references). Before starting, we recall that since we
consider $d$ to be fixed, so is the number of ideal classes, and hence
$Q$ varies in a fixed finite set, and may therefore be considered to
be fixed. Also note that
\begin{equation}\label{eq:dett}
\det(T)=dL^2M^2/4,
\end{equation}
and we seek estimates involving $\det(T)$. Thus, compared with
Kitaoka, the main difference is the dependency on $k$, which we must
keep track of. In particular, we modify and sharpen Kitaoka's method,
so that any implicit constants that appear depend only on $d$.

Because we think of $d$ as fixed, throughout the proof we drop the
subscript $d$ from the symbols $\ll, \gg, \asymp$. The reader should
not be misled into thinking that the implied constants are independent
of $d$.

Since the proof is rather technical, the reader is encouraged to
assume first that $d=4$, $M=1$ (so that there is a single class $c=c'$, and
moreover $Q=S_c=S_{c'}=1$) and also\footnote{\ This is the only case
  needed in Theorem~\ref{th-average-l-functions} for averaging the
  spin $L$-function; however, this is \emph{not} sufficient for
  Theorem~\ref{th:low-lying}, although the latter is also concerned
  only with the spinor $L$-function.} by~(\ref{eq-sclm}), $T$ is a simple diagonal matrix
$$
T=\begin{pmatrix}
L&0\\
0&L
\end{pmatrix}.
$$ 
\par
In principle, we now follow the formula for $a(P_{k,Q},T)$ which is
implicit in~\cite{Kit}. Given a system of representatives
$\mathfrak{h}$ of $\Gamma_1(\infty)\backslash\Gamma/\Gamma_1(\infty)$,
Kitaoka defines certain incomplete Poincar\'e series $H_k(M,Z)$ such
that
$$
P_{k,Q}(Z)=\sum_{M\in\mathfrak{h}}H_k(M,Z).
$$
\par
Denoting the $T$-th Fourier coefficient of $H_k(M,Z)$ by $h_k(M,T)$,
we have
$$
a(k;c,c',L,M)=a(P_{k,Q},T)=\sum_{M\in\mathfrak{h}}h_k(M,T).
$$
\par
We write
$$
M=\begin{pmatrix}
A&B\\
C&D
\end{pmatrix},
$$
where $A$, $B$, $C$ and $D$ are matrices in $M(2,\Z)$, and we now
divide the sum above depending on the rank of $C$. We denote the
component corresponding to rank $i$ by $R_i$, so that
$$
a(P_{k,Q},T)=R_0+R_1+R_2.
$$
\par
\textbf{Step 1 (rank $0$).} First of all, we consider
$R_0$. By~\cite[p. 160]{Kit} (or direct check), we have
$$
R_0=\displaystyle\sum_{\substack{U\in\GL(2,\Z)\\ UT\T{U}=Q}}1,
$$
which is $0$ unless $T$ is $\GL(2,\Z)$-equivalent to $Q$, in which
case it is equal to $|\Aut(T)|=|O(T,\Z)|$ (where $T$ is viewed as
defining a quadratic form and $O(T)$ is the corresponding orthogonal
group). In our case, looking at the determinant we find that $R_0=0$
unless $L=M=1$, and then it is also $0$ except if $c$ is
$\GL(2,\Z)$-equivalent to $c'$, and is then $|\Aut(c)|$. In other
words, we have
$$
R_0=|\Aut(c)| \cdot \delta(c,c';L;M),
$$
and hence, by definition, the remainder is therefore
\begin{equation}\label{eq-remainder-full}
R_1+R_2=L^{k-3/2}M^{k-2}A(k;c,c',L,M),
\end{equation}
and --- having isolated our main term --- we must now estimate the two
remaining ones.
\par
\medskip
\par
\textbf{Step 2 (rank $1$).} Following the computations in Kitaoka
(specifically, Lemma 4, p. 159, Lemma 1, p. 160, and up to line 2 on
p. 163 in~\cite{Kit}), but keeping track of the dependency on $k$ by
keeping the factor $Q^{3/4-k/2}$ (which Kitaoka considers as part of
his implied constant), we find that
\begin{equation}\label{rank 1}
|R_1|\ll_\eps\displaystyle\sum_{c,m\geq 1}
|T|^{k/2-3/4}|Q|^{3/4-k/2}A(m,T)m^{-1/2 + \eps}(m,c)^{1/2}
\Bigl|J_{k-3/2}\Bigl(4\pi\frac{\sqrt{|T||Q|}}{mc}\Bigr)\Bigr|
\end{equation}
where $A(m,T)$ is the number of times $T$, seen as a quadratic form,
represents $m$. 
\par
Now recall that $|Q|= d/4$ and $|T| = L^2M^2(d/4)$, and observe that
$A(m,T) = 0$ unless $L$ divides $m$ and $A(m,T) = A(m/L, S_c)$
whenever $L$ divides $m$. It follows that
$$
A(m,T)\ll_{\eps} (m/L)^\eps
$$
for any $\eps>0$. Using~\eqref{rank 1}, we get by a very rough
estimate that
\begin{align}
  |R_1|&\ll_{\eps} (LM)^{k-\frac{3}{2}}
  \displaystyle\sum_{\substack{c,m\geq 1\\ L|m}} m^{-1/2 + \eps }
  (m/L)^\eps (m,c)^{1/2} \Bigl|J_{k-3/2}\Bigl(\pi\frac{LMd}{mc}
  \Bigr)\Bigr|\nonumber\\
& \ll (LM)^{k-\frac{3}{2}}L^\eps\displaystyle\sum_{\substack{c,m_1\geq 1 }}
  m_1^{-1/2 +\eps} (m_1,c)^{1/2}
  \Bigl|J_{k-3/2}\Bigl(\pi\frac{Md}{m_1c}
\Bigr)\Bigr|\nonumber.
\end{align}
\par
Now we define
\begin{align*}
  \mathcal{R}_1 &= \sum_{\substack{c,m\geq 1 }} m^{-1/2 +\eps}
  (m,c)^{1/2} \Bigl|J_{k-3/2}\Bigl(\pi\frac{Md}{mc}
  \Bigr)\Bigr|\\
  &= \mathcal{R}_{11} + \mathcal{R}_{12}+ \mathcal{R}_{13},
\end{align*}
where $\mathcal{R}_{1i}$ corresponds to the sums restricted to
\begin{align*}
\begin{cases}
mc>\pi Md,&\text{ if } i=1\\
\pi Mdk^{-1/2}\leq mc \leq \pi MD,&\text{ if } i=2\\
mc\leq \pi Mdk^{-1/2}&\text{ if } i=3.
\end{cases}
\end{align*}
\par
For $i=1$, the argument of the Bessel function is $\leq 1$ and
by~(\ref{besselasymptotic1}), we find
\begin{align*}
\mathcal{R}_{11}\ll \frac{1}{\Gamma(k-3/2)}
\sum_{mc>\pi Md}{m^{-1/2+\eps}(m,c)^{1/2} \Bigl(\frac{\pi Md}{mc}\Bigr)^{k-3/2}}.
\end{align*}
\par
We can replace the exponent $k-3/2$ in the sum with any exponent
$1+\delta$, for small $0<\delta\leq 1$ (since $k\geq 6$ anyway), and
then we can remove the summation condition, observing that the double
series is then convergent, and obtain
$$
\mathcal{R}_{11}\ll M^{1+\eps}k^{-E},
$$
(taking $\delta$ small enough in terms of $\eps$) for any $E\geq 1$
and $\eps>0$, where the implied constant depends on $E$, $\eps$ and
$d$.
\par
For $i=2$, we use~(\ref{besselasymptotic1}) and find that
$$
\mathcal{R}_{12} \ll 
\frac{k^{k/2-3/4}}{\Gamma(k-3/2)}
\sum_{cm\leq \pi Md}{ 
m^{-1/2 +\eps} (m,c)^{1/2}}\\
\ll M^{1+\eps}k^{-E}
$$
for $E\geq 1$ and $\eps>0$ again (by summing over $c$ first and then
over $m$, and by Stirling's formula).
\par
Finally, using now~\eqref{besselasymptotic2}, we have
$$
  \mathcal{R}_{13} \ll_\eps k^{-1/3} \sum_{cm<\pi Mdk^{-1/2}}
  m^{-1/2 +\eps} (m,c)^{1/2} 
   \ll k^{-1/3-1/2} M^{1+\eps}=k^{-5/6}M^{1+\eps},
$$
(summing as for $\mathcal{R}_{12}$).
\par  
It follows that
$$
\mathcal{R}_1 = \mathcal{R}_{11} + \mathcal{R}_{12} + \mathcal{R}_{13}
\ll M^{1+\eps}k^{-\frac{5}{6}}
$$
for any $\eps>0$, and so the contribution of rank $1$ is bounded by
\begin{equation}\label{eq-contrib-rank1}
|R_1|\ll
  (LM)^{k-\frac{3}{2}}L^\eps M^{1+\eps}k^{-\frac{5}{6}}
\end{equation}
for any $\eps>0$, where the implied constant depends only on $d$ and
$\eps$.
\par
\medskip
\par
\textbf{Step 3 (rank $2$).} Finally, we deal with the $R_2$ term,
which is much more involved. The relevant set of matrices $M$ is
given by
$$
M\in \Bigl\{\begin{pmatrix} \star & \star \\
C&D
\end{pmatrix}\Bigr\}\subset \Sp(4,\Z)
$$
where $|C|\not=0$ and $D$ is arbitrary modulo $C$. Denoting
$$
M_2^*(\Z)=\{C\in M_2(\Z)\,\mid\, |C|\not=0\},
$$
we have then
$$
R_2 = \sum_{C\in M_2^*(\Z)}
\sum_{D\mods{C}}{h_k(M,T)}.
$$
\par
The inner sum was computed by Kitaoka~\cite[p. 165, 166]{Kit}. To
state the formula, let
$$
P=P(C):= TQ[\T{C^{-1}}]=T(\T{C^{-1}})QC^{-1},
$$ 
and let
\begin{equation}
0<s_1\leq s_2
\end{equation}
be such that $s_1^2$, $s_2^2$ are the eigenvalues of the positive
definite matrix $P$. Then Kitaoka proved that
\begin{equation}\label{e:step1}
\sum_{D\mods{C}}{h_k(M,T)}
=\frac{1}{2\pi^4}\Bigl(\frac{|T|}{|Q|}\Bigr)^{k/2-3/4}
|C|^{-3/2}K(Q,T;C)\mathcal{J}_k(P(C)),
\end{equation}
where $K(Q,T;C)$ is a type of matrix-argument Kloosterman sum
(see~\cite[\S 1, p. 150]{Kit} for the precise definition, which we do
not need here), and\footnote{We have made the change of variable
  $t=\sin(\theta)$ for convenience.}
$$
\mathcal{J}_k(P)=\int_0^{\pi/2}{
J_{k-3/2}(4\pi s_1\sin\theta)
J_{k-3/2}(4\pi s_2\sin\theta)
\sin\theta d\theta
}.
$$
\par
We note that
$$
|P|=|T||Q||C|^{-2}=(d/4)^2L^2M^2|C|^{-2}.
$$
\par
In order to exploit this formula~(\ref{e:step1}), we must handle the
sum over $C$. For this purpose, we use a parametrization of
$M_2^*(\Z)$ in terms of principal divisors: any $C\in M_2^*(\Z)$
can be written uniquely
\begin{equation}
\label{cosetreps}
C=U^{-1} 
\begin{pmatrix}
c_1&0\\
0&c_2
\end{pmatrix}
V^{-1}
\end{equation}
where
$$
1\leq c_1,\quad c_1\mid c_2,\quad U\in \GL(2,\Z)\text{ and } V \in
\SL(2,\Z)/\Gamma^0(c_2/c_1),
$$
where $\Gamma^0(n)$ denotes the congruence subgroup of $\SL(2,\Z)$
(conjugate to $\Gamma_0(n)$) consisting of matrices
$$
\begin{pmatrix}a&b\\c&d\end{pmatrix}
$$ 
with $n\mid b$. Note that there is a bijection
$$
\SL(2,\Z)/\Gamma^0(n)\simeq \Pp^1(\Z/nZ),
$$
(this is denoted $S(n)$ in~\cite{Kit}) and in particular
\begin{equation}\label{eq-psi}
|\SL(2,\Z)/\Gamma^0(n)|=n\prod_{p\mid n}{(1+p^{-1})}\ll n^{1+\eps}
\end{equation}
for any $\eps>0$.
\par
We will first consider matrices where the last three parameters
$\uple{c}=(c_1,c_2,V)$ are fixed, subject to the conditions above. The
set of such triples is denoted $\mathcal{V}$, and for each
$\uple{c}\in\mathcal{V}$, we fix (as we can) a matrix $U_1 \in
\GL(2,\Z)$ such that the matrix
$$
A(\uple{c})=A:= T\Bigl[V
\begin{pmatrix}
c_1&
\\
&c_2
\end{pmatrix}^{-1}
U_1\Bigr]
$$ 
is Minkowski-reduced. This matrix is conjugate to a diagonal matrix
$H=H(\uple{c})$ of the form
$$
H= \begin{pmatrix}a&\\&c\end{pmatrix}
$$
with $a \le c$. Computing determinants and using the fact that $A$ is
Minkowski-reduced, we note also that we have
\begin{equation}\label{eqacs1s2}
(d/4) \frac{L^2M^2}{c_1^2c_2^2} = ac \asymp s_1^2s_2^2 =(d/4)^2 \frac{L^2M^2}{c_1^2c_2^2}
\end{equation}
(we recall again that $d$ is assumed to be fixed).
\par
For fixed $\uple{c}\in \mathcal{V}$, the set of matrices $C\in
M_2^*(\Z)$ corresponding to $\uple{c}$ can be parameterized in the
form
$$
C = U^{-1}U_1^{-1}\begin{pmatrix}
c_1&\\&c_2
\end{pmatrix}
V^{-1},
$$ 
where $U$ varies freely over $\GL(2,\Z)$ (this is a simple change of
variable of the last parameter $U\in \GL(2,\Z)$
in~\eqref{cosetreps}). As shown in~\cite[p. 167]{Kit}, for any such
$C$ associated to $\uple{c}$, we have also
$$
|P| \asymp |A|,\quad \Tr(P) \asymp \Tr(A[U])=\Tr(H[U]).
$$
\par
We can now start estimating. First, for a given $C$ parametrized by
$(U,\uple{c})$, Kitaoka proved (see~\cite[Prop. 1]{Kit}) that the
Kloosterman sum satisfies
$$
|K(Q,T;C)| \ll c_1^2 c_2^{1/2 + \eps}(c_2 , T[v])^{1/2},
$$ 
for any $\eps>0$, where $v$ is the second column of $V$ and the
implied constant depends only on $\eps$. Hence by~(\ref{e:step1}), we
obtain
$$
\sum_{D\mods{C}}h_k(M,T)
\ll (LM)^{k-3/2} c_1^{1/2} c_2^{-1 +\eps}(c_2 , T[v])^{1/2}
|\mathcal{J}_k(P(C))|.
$$
\par
In order to handle the Bessel integral $\mathcal{J}_k(P(C))$, we will
partition $M_2^*(\Z)$ in three sets $\mathcal{C}_1$, $\mathcal{C}_2$,
$\mathcal{C}_3$, according to the relative sizes of the values $s_1$
and $s_2$ for the corresponding invariants $\uple{c}$. These can be
determined from the size of $\Tr(P)$ and $|P|$; precisely, we let
\begin{gather*}
\mathcal{C}_1=\{C\,\mid\, \Tr(P)<1\},\\
\mathcal{C}_2=\{C\,\mid\, \Tr(P)\geq \max(2|P|,1)\},\\
\mathcal{C}_3=\{C\,\mid\, 1\leq \Tr(P)<2|P|\},
\end{gather*}
and we further denote by $\mathcal{C}_i(\uple{c})$ the subsets of
$\mathcal{C}_i$ where $C$ is associated with the invariants
$\uple{c}=(c_1,c_2,V)$. The following lemma gives the rough size of
these sets, or a weighted version that is needed below:

\begin{lemma}
With notation as above, for any $\uple{c}=(c_1,c_2,V)$, we have
\begin{gather}
|\mathcal{C}_1(\uple{c})|\ll
(ac)^{-1/2-\eps},
\label{eq-size1}
\\
\sum_{C\in\mathcal{C}_2(\uple{c})}{
|A|^{1+\delta}(\Tr(A[U]))^{-5/4-\delta}}
\ll
\begin{cases}
(ac)^{1/2 +\delta - \eps}&\text{ if } ac<1\\
(ac)^{1/4 + \eps}&\text{ if } ac\geq 1,
\end{cases}
\label{eq-size2}
\\
|\mathcal{C}_3(\uple{c})|\ll
(ac)^{1/2+\eps},
\label{eq-size3}
\end{gather}
for any $\eps>0$ and $\delta> 0$ in the second, where the implied
constants depend on $\delta$ and $\eps$.
\end{lemma}

\begin{proof}
All these are proved by Kitaoka. Precisely:
\begin{itemize}
\item The bound~(\ref{eq-size1}) comes from~\cite[p. 167]{Kit}, using
  the fact that in that case we have $a\ll 1$;
\item The bound~(\ref{eq-size2}) comes from the arguments
  of~\cite[p. 168, 169]{Kit} (note that in that case the summation set
  is infinite); to be more precise, Kitaoka argues with what amounts
  to taking
$$
\delta=k/2-7/4,\text{ so that } k/2-3/4=1+\delta,\quad 
(1-k)/2=-5/4-\delta,
$$
but the only information required (up to~\cite[p. 169, line 10]{Kit})
is the sign and the value of the sum of the two exponents
$$
k/2-3/4+(1-k)/2=-1/4=1+\delta + (-5/4-\delta)
$$
(this is used in~\cite[p. 168, line -12]{Kit}). Hence Kitaoka's
argument applies for $\delta>0$.
\item The bound~(\ref{eq-size3}) comes from~\cite[p. 168]{Kit}, using
  the fact that in that case we have $c\gg 1$.
\end{itemize}
\end{proof}

As shown also by Kitaoka, we have the following crucial localization
properties (see~\cite[p. 166]{Kit}):
\begin{enumerate}
\item If $C\in\mathcal{C}_1$, then $s_1 \le 1$ and $s_2\le 1$;
\item If $C\in\mathcal{C}_2$, then $s_1 \le 1$ and $s_2 \gg 1$, with
  absolute implied constant;
\item If $C\in \mathcal{C}_3$, then $s_1 \gg 1$ and $s_2 \gg 1$, with
  absolute implied constants.
\end{enumerate}
\par
Now, by breaking up the sum over $C$ in $R_2$ according to the three
subsets $\mathcal{C}_i$, we can write
$$
|R_2| \ll  R_{21} + R_{22} + R_{23}
$$
where, for $i=1$, $2$, $3$, and any fixed $\eps>0$, we have
$$
R_{2i} \ll
(LM)^{k-\frac{3}{2}}\sum_{\uple{c}\in\mathcal{V}}c_1^{\frac{1}{2}}
c_2^{-1 + \eps} (c_2,T[v])^{1/2} \mathcal{R}_{2i}(\uple{c}),
$$ 
for any $\eps>0$ with
$$
\mathcal{R}_{2i}(\uple{c})=
\sum_{C\in\mathcal{C}_i(\uple{c})}{|\mathcal{J}_k(P(C))|},
$$
the implied constant depending only on $d$ and $\eps$.
\par
Accordingly, we study each of $R_{21}$, $R_{22}$, $R_{23}$ separately.
\par
\medskip
\par
\textbf{-- Estimation of $R_{21}$.}  Since we have $s_1 \le 1, s_2 \le
1$, we use~\eqref{besselasymptotic1}; using the superexponential
growth of the Gamma function, we obtain easily
$$
|\mathcal{J}_k(P(C)) | \ll_\eps \frac{(s_1s_2)^{2+\delta}}{2^k}
$$
for $C\in\mathcal{C}_1$ and any fixed $\delta>0$. On the other hand,
by~(\ref{eq-size1}), we have
$$
|\mathcal{C}_1(\uple{c})|\ll (ac)^{-\frac{1}{2} - \epsilon} \ll
(s_1s_2)^{-1-2\eps},
$$
for any $\epsilon>0$, and taking it small enough we obtain
\begin{equation}\label{proto}
  \mathcal{R}_{21}(\uple{c}) \ll
  \frac{(s_1s_2)^{1+\delta}}{2^k} \ll  
  (LM)^{1+\delta}\frac{(c_1c_2)^{-1-\delta}}{2^k},
\end{equation}
for any fixed $\delta>0$. For fixed $c_1$, $c_2$ first, we have
\begin{align*}
  \sum_{V\in \SL(2,\Z)/\Gamma^0(c_2/c_1)}c_1^{\frac{1}{2}} c_2^{-1 +
    \eps}(c_2 , T[v])^{1/2} \mathcal{R}_{21}(\uple{c}) 
  & \ll \frac{(LM)^{1 +\delta}}{2^k} \sum_{V}c_1^{-\frac{1}{2} -
    \delta}
  c_2^{-2 - \delta +\eps}(c_2 , T[v])^{1/2}
\end{align*}
from which one deduces easily
$$
 \sum_{V\in \SL(2,\Z)/\Gamma^0(c_2/c_1)}c_1^{\frac{1}{2}} c_2^{-1 +
    \eps}(c_2 , T[v])^{1/2} \mathcal{R}_{21}(\uple{c}) 
\ll \frac{(LM)^{1 + \delta}}{2^k}(c_1c_2)^{-1 - \delta+\eps}
  (c_2/c_1 , LM^2)^{1/2}
$$
for any $\delta>0$, possibly different than before
(using~(\ref{eq-psi}) and~\cite[Prop. 2]{Kit} to handle the gcd; the
exponent of $c_1$ was worsened by $1/2$ to facilitate the use of this
lemma).
\par
Writing $c_2 = nc_1$, with $n\geq 1$, we can finally sum over $c_1$
and $n$; the resulting series converge for $\delta>0$ and we obtain
$$
R_{21} \ll 2^{-k}(LM)^{k-\frac{3}{2}+1+\delta}
\sum_{c_1,n\geq 1} c_1^{-2}n^{-1-\delta+\eps}(n , LM^2)^{1/2},
$$
and therefore by taking, e.g., $\delta=2\eps$ (and changing notation),
we derive
\begin{equation}\label{eq-contrib-r21}
R_{21}\ll (LM)^{k-\frac{3}{2}} (LM)^{1 + \eps} k^{-E}
\end{equation}
for any $\eps>0$ and $E\geq 1$, where the implied constant depends on
$d$, $E$ and $\eps$.
\par
\medskip
\par
\textbf{-- Estimation of $R_{22}$.} We treat the $R_{22}$ term next.
Using~\eqref{besselasymptotic1} for the Bessel function involving
$s_1$ and~(\ref{besselasymptotic3}) for the one involving $s_2$, and
using the fact that
$$
\Tr(A[U])\asymp \Tr(P)=s_1^2+s_2^2\asymp s_2^2
$$
for this term, it is easy to check that
$$
|\mathcal{J}_k(P(C)) | \ll \frac{1}{\Gamma(k-3/2)}
|A|^{\frac{k}{2} - \frac{3}{4}}
(\Tr A[U])^{\frac{1-k}{2}}.
$$
\par
If we write
$$
|A|^{\frac{k}{2} - \frac{3}{4}}
(\Tr A[U])^{\frac{1-k}{2}}
=
|A|^{1+\delta}(\Tr A[U])^{-5/4-\delta}
\Bigl(\frac{|A|}{\Tr(A[U])}\Bigr)^{k/2-7/4-\delta}
$$
for any fixed $\delta>0$, and observe that
$$
\frac{|A|}{\Tr(A[U])}\asymp \frac{ac}{s_2^2}\asymp s_1^2\ll 1,
$$
it follows using the super-exponential growth of the Gamma function
that
$$
\mathcal{R}_{22}(\uple{c})
\ll 
2^{-k}\sum_{C\in\mathcal{C}_2(\uple{c})} |A|^{1+\delta} (\Tr
A[U])^{-\frac{5}{4}-\delta}
$$
for any fixed $\delta>0$. By~(\ref{eq-size2}), we have
$$
\mathcal{R}_{22}(\uple{c})
\ll 2^{-k}\times 
\begin{cases}(ac)^{\frac{1}{2}+ \delta- \epsilon} & \text{ if } {ac <1},
  \\ (ac)^{ \frac{1}{4} + \epsilon} & \text{ if } {ac \ge 1},
\end{cases}
$$
for any $\epsilon>0$ and $\delta>0$.
\par
We take $\epsilon=\delta/2$ and using~(\ref{eqacs1s2}), we deduce that
\begin{multline*}
  R_{22} \ll \frac{(LM)^{k - \frac{3}{2}}}{2^k}\bigg
  (\sum_{c_1c_2>d^{1/2}LM/4}
  \left(\frac{LM}{c_1c_2}\right)^{1+\delta} \sum_{V\in
    \SL(2,\Z)/\Gamma^{0}(c_2/c_1)}c_1^{\frac{1}{2}} c_2^{-1 +\eps}(c_2
  , T[v])^{1/2} \\
  + \sum_{c_1c_2\leq d^{1/2}LM/4}
  \left(\frac{d^{1/2}LM}{2c_1c_2}\right)^{\frac{1}{2}+\delta} \sum_{V\in
    \SL(2,\Z)/\Gamma^{0}(c_2/c_1)}c_1^{\frac{1}{2}} c_2^{-1 +\eps}
(c_2 , T[v])^{1/2} \bigg)
\end{multline*}
for any $\delta>0$. In the second sum, we can write trivially
$$
\left(\frac{d^{1/2}LM}{2c_1c_2}\right)^{\frac{1}{2}+\delta} \le
\left(\frac{d^{1/2}LM}{2c_1c_2}\right)^{1+\delta}
\ll
\left(\frac{LM}{c_1c_2}\right)^{1+\delta},
$$
so we end up with
$$
R_{22} \ll \frac{(LM)^{k - \frac{3}{2}}}{2^k}
\sum_{\substack{c_1|c_2}} \left(\frac{LM}{c_1c_2}\right)^{1+\delta}
\sum_{V\in \SL(2,\Z)/\Gamma^{0}(c_2/c_1)}c_1^{\frac{1}{2}} c_2^{-1 +
  \eps}(c_2 , T[v])^{1/2}
$$
and now, using the same type of arguments leading from~\eqref{proto}
to~(\ref{eq-contrib-r21}), we see that
\begin{equation}\label{eq-contrib-r22}
R_{22} \ll (LM)^{k - 3/2}(LM)^{1+\eps}k^{-E}
\end{equation}
for any $\eps>0$ and $E\geq 1$, the implied constant dependind on $d$,
$E$ and $\eps$.
\par
\medskip
\par
\textbf{-- Estimation of $R_{23}$.}  Recall that we have $1 \ll s_1
\le s_2$ for $C\in\mathcal{C}_3$. We estimate the Bessel integral
using
$$
|\mathcal{J}_k(P)|
\leq \Bigl(\int_{M_1}+
\int_{M_2}\Bigr)
\Bigl|
J_{k-3/2}(4\pi s_1\sin\theta)
J_{k-3/2}(4\pi s_2\sin\theta)
\sin\theta\Bigr| d\theta
$$
where 
\begin{gather*}
  M_1 = \{\theta\in [0,\pi/2]\,\mid\, 4\pi s_1\sin\theta \le 1 \},
  \\
  M_2 = \{ \theta\in [0,\pi/2]\,\mid\, 1 \le 4\pi s_1\sin\theta , \ 1
  \le 4\pi s_2\sin\theta \}.
\end{gather*}
\par
In the first we use~\eqref{besselasymptotic1} and the
super-exponential growth of the Gamma function to write
$$
J_{k-3/2}(4\pi s_1\sin\theta) \ll 2^{-k}s_1^\delta, \quad
J_{k-3/2}(4\pi s_2\sin\theta) \ll 1\ll s_2^\delta,
$$
for any $\delta>0$, and in the second we use the
estimate~(\ref{besselasymptotic2}) to get
$$
J_{k-3/2}(4\pi s_1\sin\theta) J_{k-3/2}(4\pi s_2\sin\theta)\ll
k^{-2/3},
$$
so that
$$
|\mathcal{J}_k(P(C)) | \ll k^{-2/3}+2^{-k}(s_1s_2)^{\delta}\ll
k^{-2/3}(s_1s_2)^{\delta}
$$ 
for any $\delta>0$. It follows that
$$
\mathcal{R}_{23}(\uple{c}) \ll
k^{-2/3}\sum_{C\in\mathcal{C}_3(\uple{c})}(s_1s_2)^\delta,
$$ 
which, by~(\ref{eq-size3}) with, e.g., $\eps=\delta$, gives
$$
\mathcal{R}_{23}(\uple{c}) \ll
k^{-2/3}(LM)^{1+\delta}(c_1c_2)^{-1-\delta}
$$
for any $\delta>0$. Then the same argument as that
following~\eqref{proto} is used to sum over the parameters $\uple{c}$,
and to deduce
\begin{equation}\label{eq-contrib-r23}
R_{23} \ll (LM)^{k - 3/2}(LM)^{1 + \eps}k^{-E}
\end{equation}
for any $\eps>0$ and $E\geq 1$, the implied constant depending on $d$,
$E$ and $\eps$-
\par
Summarizing, we have
$$
(LM)^{k-3/2}A(k;c,c',L,M)=R_1+R_2\ll R_1+R_{21}+R_{22}+R_{23},
$$
and putting together the estimates~(\ref{eq-contrib-rank1}),
(\ref{eq-contrib-r21}), (\ref{eq-contrib-r22}), and
(\ref{eq-contrib-r23}), we find that we have proved the estimate
$$
A(k;c,c',L,M)\ll L^{1 + \eps}M^{3/2 + \eps}k^{-2/3}
$$
for $\eps>0$, which was our goal.
\end{proof}

\begin{remark}
  For later investigations, it may be worth pointing out that the
  limitation on the error term, as a function on $k$, arises only from
  the contributions $R_1$ and (the second part of) $R_{23}$. All other
  terms decay faster than any polynomial in $k$ as $k\ra +\infty$.
\end{remark}

\subsection{A quasi-orthogonality relation for Siegel modular forms}
\label{s:petersson}

We now put together the results of the previous sections. For every
$k\geq 1$, we fix a Hecke basis $\cusp_k^*$ of $\cusp_k$. Fix the data
$(d,\Lambda)$ as in Section~\ref{sec:statement} and let
$\weight_{k,d,\Lambda}^F$ be as defined there; accordingly we have
measures $\nu_{sop,k}$ defined for every finite set of primes $\sop$ and
weight $k\geq 1$ using suitable average over $F\in \cusp_k^*$.
\par
Our main result in this section is:

\begin{proposition}\label{pr:petersson}
  Let $\sop$ be a finite set of primes, and $l=(l_p)$, $m=(m_p)$ be
  $\sop$-tuples of non-negative integers. Put
$$
L=\prod_{p\in \sop}{p^{l_p}},\quad\quad M=\prod_{p\in \sop}{p^{m_p}}.
$$
\par
Then we have
$$
\int_{X_{\sop}}{\prod_{p\in \sop}{ U^{l_p,m_p}_p(x_p)}d\nu_{\sop,k}}=
\sum_{F \in \cusp_{k}^\ast} \weight^F_{k,d,\Lambda} \prod_{p\in \sop}
U^{l_p,m_p}_{p}(a_{p}(F),b_{p}(F)) \longrightarrow \delta(l;m)
$$
as $k \rightarrow \infty$ over the \emph{even} integers, where
$$
\delta(l;m)=\begin{cases}
1&\text{ if } L=M=1,\text{ i.e. all $l_p$ and $m_p$ are $0$},\\
0&\text{ otherwise}.
\end{cases}
$$
\par
More precisely, for any even $k$ we have
\begin{equation}\label{eq:ortho-quant}
\sum_{F \in \cusp_{k}^\ast} \weight^F_{k,d,\Lambda} \prod_{p\in \sop}
U^{l_p,m_p}_{p}(a_{p}(F),b_{p}(F))=
\delta(l;m)+O\left(\frac{L^{1 +  \eps}M^{3/2 +  \eps}}{k^{\frac{2}{3} }}\right),
\end{equation}
for any $\eps>0$, where the implied constant depends only on $d$ and
$\eps$.
\end{proposition}

We will first prove a lemma which is easy, but where the distinction
between $\SL(2,\Z)$ and  $\GL(2,\Z)$-equivalence of quadratic forms is
important. 

\begin{lemma}\label{autcsum}
For $c$, $c'\in \Cl_d$, put 
$$
\delta(c,c')= \begin{cases}
  1 &\text{ if } $c$ \text{ is $\GL(2,\Z)$-equivalent to $c'$},\\
    0 & \text{ otherwise. }
\end{cases}
$$ 
\par
Then we have
$$
\sum_{c,c'\in\Cl_d} 
\Lambda(c)\overline{\Lambda(c')}
\delta(c,c')|\Aut(c)| 
= \frac{2h(-d)w(-d)}{d_\Lambda}
$$
where
$$
d_\Lambda =
\begin{cases} 1 & \text{ if } \Lambda^2 =1,\\
  2 & \text{ otherwise. }
\end{cases}
$$
\end{lemma}

\begin{proof} 
  Let $H \subset \Cl_d$ be the group of $2$-torsion elements. The
  classes $c'$ which are $\GL(2,\Z)$-equivalent to a given class $c$
  are $c$ and $c^{-1}$, hence there are either one or two, depending
  on whether $c\in H$ or not. Similarly, $|\Aut(c)|$ (which is the
  order of $\GL(2,\Z)$-automorphisms of a representative of $c$)
  equals either $2w(-d)$ or $w(-d)$, depending on whether $c$ lies in
  $H$ or not.

Therefore, we have
\begin{align*}
  \sum_{c,c'\in\Cl_d} \Lambda(c)\overline{\Lambda(c')}
  \delta(c,c')|\Aut(c)| &=
 \sum_{c\in\Cl_d} \Lambda(c)|\Aut(c)| 
\sum_{c'}\overline{\Lambda(c')} \delta(c,c')
\\
&=\sum_{c\in H}{|\Aut(c)|\Lambda(c)^2}
+
\sum_{c\notin H}{|\Aut(c)|\Lambda(c)(\Lambda(c)+\Lambda(c^{-1}))}
\\
&=
w(-d)\sum_{c}{(1+\Lambda^2(c))}
\end{align*}
by writing $2=1+\Lambda(c^2)=1+\Lambda^2(c)$ when $c\in H$.
The result follows immediately.
\end{proof}

Now we come to the proof of Proposition~\ref{pr:petersson}.

\begin{proof}
  For brevity, we drop the subscripts $d$ and $\Lambda$ here. For
  $F\in \cusp_k^*$, by definition of $\weight_k^F$, we have
$$
\weight^F_{k}\prod_{p\in \sop} U^{l_p,m_p}_{p}(a_{p}(F),b_{p}(F))
=\frac{4c_k d_\Lambda(d/4)^{\frac{3}{2}-k}}{w(-d)h(-d)}
  \frac{|a(d, \Lambda;F)|^2} {\langle F, F \rangle}\prod_{p\in \sop}
  U^{l_p, m_p}_{p}(a_{p}(F),b_{p}(F)).
$$
\par
Now, we write
$$
|a(d, \Lambda;F)|^2=a(d,\Lambda;F)\overline{a(d,\Lambda;F)}
$$
and using~(\ref{eq:alf}) to express the first term, we get
$$
\frac{|a(d, \Lambda;F)|^2} {\langle F, F \rangle}\prod_{p\in \sop}
U^{l_p, m_p}_{p}(a_{p}(F),b_{p}(F)) =
\frac{\overline{a(d,\Lambda;F)}}{\langle F,F\rangle} 
\sum_{c\in\Cl_d}{\overline{\Lambda(c)}
a(F,S_c) \prod_{p\in \sop} U^{l_p,m_p}_{p}(a_{p}(F),b_{p}(F))}.
$$
\par
Now Theorem~\ref{t:relation} applies to transform this into
\begin{align*} 
  \frac{|a(d, \Lambda;F)|^2} {\langle F, F \rangle}\prod_{p\in \sop}
  U^{l_p, m_p}_{p}(a_{p}(F),b_{p}(F)) &= \frac{|\Cl_d|}{|\Cl_d(M)|} \frac{L^{\frac{3}{2}-k}
    M^{2-k}\overline{a(d,\Lambda;F)}}{\langle F,F\rangle}
  \sum_{c'\in \Cl_d(M)} {\overline{\Lambda(c')}a(F,S_{c'}^{L,M})}
\\
&=\frac{|\Cl_d| L^{3/2-k}M^{2-k}}{|\Cl_d(M)|}
\sum_{\substack{c\in \Cl_d\\ c'\in \Cl_d(M)} }{\Lambda(c)\overline{\Lambda(c')} \cdot 
\frac{\overline{a(F,S_{c})}a(F,S_{c'}^{L,M})}{\langle F,F\rangle}},
\end{align*}
after expanding $\overline{a(d,\Lambda;F)}$ using its definition. We
are now reduced to a quantity involving only Fourier coefficients.
\par
We then apply the basic property of the Poincar\'e
series~(\ref{eq:poincare-basic}) to express these Fourier coefficients
in terms of inner product with Poincar\'e series: we have
\begin{align*}
  \overline{\langle F,P_{k, S_{c}} \rangle} &= 8c_k \left(\frac{d}{4}
  \right)^{-k+3/2}
\overline{a(F,S_{c})},\\
  \langle F,P_{k, S_{c'}^{L,M}}  \rangle &= 8c_k
  (LM)^{-2k+3}\left(\frac{d}{4} \right)^{-k+
    3/2}a(F,S_{c'}^{L,M})
\end{align*}
for $c\in \Cl_d$, $c'\in \Cl_d(M)$, and multiplying out with the normalizing
constants, we get
$$
\weight^F_{k}\prod_{p\in \sop} U^{l_p,m_p}_{p}(a_{p}(F),b_{p}(F)) =
\frac{M^{k-1}L^{k-3/2}d_\Lambda(d/4)^{k-\frac{3}{2}} }{16 c_kw(-d) |\Cl_d(M)|}
\sum_{\substack{c\in \Cl_d\\ c'\in \Cl_d(M)}}
\Lambda(c)\overline{\Lambda(c')}
\frac{\overline{\langle  F,
  P_{k, S_{c}} \rangle} \langle F,P_{k, S_{c'}^{L,M}} \rangle}{\langle
  F, F \rangle}.
$$
for every $F\in\cusp_k^*$.
\par
We now sum over $F$ and exchange the summation to average over $F$
first. Since $\{F/\|F\|\}$ is an orthonormal basis of the vector space
$\cusp_k$, we have
$$
\sum_{F \in \cusp_k^\ast} \frac{1}{\|F\|^2}
\overline{\langle F,P_{k,S_{c}}\rangle}
\langle F,P_{k, S_{c'}^{L,M}} \rangle
=
\langle P_{k, S_{c}}, P_{k, S_{c'}^{L,M}} \rangle.
$$
\par
Now, according to~(\ref{eq:poincare-basic}) again, we have
$$
 \langle P_{k,  S_{c}},P_{k, S_{c'}^{L,M}} \rangle
= 8c_k\Bigl(\frac{dL^2M^2}{4}\Bigr)^{-k+\frac{3}{2}}
a(k;c,c',L,M)
$$
where $a(k;c,c',L,M)$ denotes, as before, the $S_{c'}^{L,M}$-th Fourier
coefficient of the Poincar\'e series $P_{k, S_{c}}$. Applying this
and the formal definition
$$
a(k;c,c',L,M)=|\Aut(c)|\delta(c,c';L;M)+L^{k-3/2}M^{k-2}
A(k;c,c',L,M),
$$
as in Proposition~\ref{proppetersson}, we obtain first, using
Lemma~\ref{autcsum} that
\begin{multline*}
\sum_{F \in \cusp_k^\ast} \weight^F_{k}\prod_{p\in \sop} U^{l_p,
  m_p}_{p}(a_{p}(F),b_{p}(F))=
\delta(l;m)+\\
L^{-k+3/2}M^{-k+2}\frac{d_\Lambda}{2h(-d)w(-d)|\Cl_d(M|}
\sum_{\substack{c\in \Cl_d\\ c'\in \Cl_d(M)}}
\Lambda(c)\overline{\Lambda(c')}
A(k;c,c',L,M),
\end{multline*}
and then Proposition~\ref{proppetersson} and
Proposition~\ref{proppetersson2} lead immediately to the desired
result.
\end{proof}

\begin{remark}
  In the case of cusp forms on $\SL(2,\Z)$ and its congruence
  subgroups, one can write the Petersson formula in a way which is
  suitable for further transformations (with ``off-diagonal terms''
  involving Kloosterman sums), as first investigated by Duke,
  Friedlander and Iwaniec. These are of crucial importance in, e.g.,
  the extension of the range of test functions for low-lying zeros
  in~\cite{ils}). In our case, the complexity of the analogue
  expansion (which is only implicit in Kitaoka's work) for Siegel cusp
  forms makes this a rather doubtful prospect, at least at the moment.
\end{remark}

\section{Local equidistribution}\label{equidistribution}

To pass from Proposition~\ref{pr:petersson} to a local
equidistribution result, we must understand how the test functions
considered there relate to the space of all continuous functions on
$Y_{\sop}$. This is the purpose of this section.

\subsection{Symmetric functions and polynomials}

We first observe explicitly that the Laurent polynomials
$$
U_p^{l,m}(a,b)\in \C[a,b,a^{-1},b^{-1}]
$$
of Theorem~\ref{t:sugano} are invariant under the transformations
$$
(a,b)\mapsto (b,a),\quad\quad (a,b)\mapsto (a^{-1},b),\quad\quad
(a,b)\mapsto (a,b^{-1}),
$$
which means that they can be interpreted as functions (also denoted
$U_p^{l,m}$) on the space $Y_p$ or on the set $X_p$ of unramified
principal series of $G(\Q_p)$.  We first state a simple consequence of
Proposition~\ref{pr:pol-basis}.

\begin{corollary}\label{cor:spanning}
  Let $\sop$ be a fixed finite set of primes, and let  $Y_{\sop}$ be as
  before. The linear span of the functions
$$
(a_p,b_p)_{p\in \sop}\mapsto
\prod_{p\in \sop}{
U_p^{l_p,m_p}(a_p,b_p)
},
$$
where $(l_p)$, $(m_p)$ run over non-negative integers indexed by $\sop$,
is dense in the space $C(Y_{\sop})$ of continuous functions on $Y_{\sop}$.
\end{corollary}

\begin{proof}
  By the Stone-Weierstrass Theorem, this follows immediately from
  Proposition~\ref{pr:pol-basis}, using the product structure to go
  from a single prime to a finite set of primes.
\end{proof}

The point of this, in comparison with Proposition~\ref{pr:petersson},
is of course the following fact:

\begin{proposition}\label{propplancherelmeasure}
  Let $\sop$ be any fixed finite set of primes, and let $\mu_{\sop}$
  be the associated Plancherel measure on $Y_{\sop}$, defined in the
  introduction. We have
$$
\int_{Y_{\sop}}{ \prod_{p\in \sop}{ U_p^{l_p,
  m_p}(a_{p},b_{p})}d\mu_{\sop}}
=
\begin{cases} 
1 &\text{ if } l_p = m_p=0 \text{ for all } p\in \sop,
\\ 
0 & \text{ otherwise, } 
\end{cases}
$$
for all non-negative integers $(l_p)$, $(m_p)$ indexed by primes in
$\sop$.
\end{proposition}

\begin{proof} 
  Since we work with product measures and product functions, it is
  enough to prove this for the case $n=1$. But that follows directly
  from~\cite[equation (8)]{furmarsha}.
\end{proof}

\begin{remark}
  This fact can also be proved by direct contour integration via
  Cauchy's formula using the generating function description for
  $U_{p}^{l,m}(a,b)$ (given by Theorem~\ref{t:sugano}).
\end{remark}

It is now a simple matter to conclude the proof of
Theorem~\ref{maintheorem}. 

\begin{proof}[Proof of Theorem~\ref{maintheorem}]
  Fix a finite set of primes $\sop$. Using the Weyl equidistribution
  criterion, in order to prove that $\nu_{\sop,k}$ converges weakly to
  $\mu_{\sop}$ as $k\ra +\infty$ over even integers, it suffices to show that
\begin{equation}\label{eq:todo}
\lim_{k\ra +\infty}{\int_{Y_{\sop}}\varphi((x_p))d\nu_{\sop,k}}
=\int_{Y_{\sop}}{\varphi(x)d\mu_{\sop}(x)}
\end{equation}
for all functions $\varphi$ taken from a set of continuous functions
whose linear combinations span $C(Y_{\sop})$. By
Corollary~\ref{cor:spanning}, the functions 
$$
\varphi((a_p,b_p))=
\prod_{p\in \sop}{
U_p^{l_p,m_p}(a_p,b_p)
},
$$
where $(l_p)$, $(m_p)$ are non-negative integers indexed by $\sop$, form
such a set. But for $\varphi$ of this type, the desired
limit~(\ref{eq:todo}) is obtained by combining
Proposition~\ref{propplancherelmeasure} and
Proposition~\ref{pr:petersson}. 
\par
Now, for the proof of the quantitative
version~(\ref{eq:quant-local}). First of all, we can assume that all
polynomials $\varphi_p$ are non-constant, i.e., that $d_p\geq 1$ for
each $p\in \sop$, by working with a smaller $\sop$ if necessary (and
incorporating the constant functions at a single prime).  The
polynomials $\varphi_p$ are finite linear combinations, say
$$
\varphi_p(a_p,b_p)=\sum_{0\leq l_p\leq e_p}\sum_{0\leq m_p\leq f_p}{
  \hat{\varphi}_p(l_p,m_p)U_p^{l_p,m_p}(a,b)}
$$
of the basis polynomials $U^{l_p,m_p}_p$, for some $e_p$, $f_p\geq 0$
with $\max(e_p,f_p)\geq 1$.
\par
Taking the product of these expressions over $\sop$, summing over $F$,
and using~(\ref{eq:ortho-quant}) we get
\begin{equation}\label{eq:quantrem}
\sum_{F\in \cusp_{2k}^*}{\weight_{2k}^F \varphi((a_p(F),b_p(F))_{p\in
    S})}
=\int_{Y_{\sop}}{\varphi(x)d\mu_{\sop}(x)}+k^{-2/3}R
\end{equation}
where we see that the remainder $R$ can be bounded by
$$
|R|\ll
\sum_{L\mid L_{\varphi}}\sum_{M\mid M_{\varphi}}
L^{1+\eps}M^{3/2+\eps}
\prod_{p\in \sop}{|\hat{\varphi}_p(v_p(L),v_p(M))|}
$$
for any $\eps>0$, where the implied constant depends only on $\eps$
and
$$
L_{\varphi}=\prod_{p\in \sop}{p^{e_p}},\quad\quad
M_{\varphi}=\prod_{p\in \sop}{p^{f_p}}.
$$
\par
The coefficients in the expansion are obtained as inner products
$$
\hat{\varphi}_p(l,m)=\langle \varphi_p, U_p^{l,m}\rangle
$$
in $L^2(Y_p,d\mu_p)$, by orthogonality of the polynomials
$U_p^{l,m}$. Since the underlying measure $\mu_p$ is a probability
measure supported on the tempered subset $X\subset Y_p$, those
coefficients may be bounded by
$$
|\hat{\varphi}_p(v_p(L),v_p(M))|\leq
\|\varphi_pU_p^{l_p,m_p}\|_{\infty}
\leq 
C(v_p(L)+3)^3(v_p(M)+3)^3
\|\varphi_p\|_{\infty},
$$
by Lemma~\ref{lm:bounds-ulm}. Therefore, we get the estimate
$$
|R|\ll L^{1+\eps}_{\varphi}M^{3/2+\eps}_{\varphi} \eta(LM) \prod_{p\in
  S}{\|\varphi_p\|_{\infty}}=
L^{1+\eps}_{\varphi}M^{3/2+\eps}_{\varphi} \eta(LM)
\|\varphi\|_{\infty}
$$
where $\eta(n)$ is the multiplicative function such that
$$
\eta(p^{\nu})=C(\nu+1)^6
$$
for $p$ prime and $\nu\geq 0$ (here we use the fact that
$\max(e_p,f_p)\geq 1$ for each $p$). This is a divisor-like function,
i.e., it satisfies
$$
\eta(n)\ll n^{\eps}
$$
for any $\eps>0$, the implied constant depending only on
$\eps$. Therefore we get
$$
|R|\ll L^{1+\eps}_{\varphi}M^{3/2+\eps}_{\varphi}
\|\varphi\|_{\infty},
$$
for any $\eps>0$, where the implied constant depends only on $\eps>0$.
\par
To derive~(\ref{eq:quant-local}), we observe that, by the second part of
Proposition~\ref{pr:pol-basis}, the linear decomposition of
$\varphi_p$ holds with
$$
e_p+2f_p\leq d_p,
$$
where $d_p$ is the total degree of $\varphi_p$ as polynomial in
$(a+a^{-1},b+b^{-1})$. Thus, at the cost of worsening the factor
involving $M$, we obtain~(\ref{eq:quant-local}).
\end{proof}

\begin{remark}\label{rm:stronger}
  The proof shows that if we know that the factors $\varphi_p$ are
  combinations of polynomials $U_p^{l,m}$ with $l\leq l_p$, $m\leq
  m_p$, we have the stronger estimate
$$
\sum_{F\in \cusp_{2k}^*}{\weight_{2k}^F \varphi((a_p(F),b_p(F))_{p\in
    S})}= \int_{Y_{\sop}}{\varphi(x)d\mu_{\sop}(x)} 
+O\Bigl(
k^{-2/3}L^{1+\eps}M^{3/2+\eps}\|\varphi\|_{\infty} \Bigr)
$$
for any $\eps>0$, where
$$
L=\prod_{p\in \sop}{p^{l_p}},\quad\quad M=\prod_{p\in \sop}{p^{m_p}}.
$$
\end{remark}

\section{Applications}\label{applications}

We now gather some applications of the local equidistribution
theorem. To emphasize the general principles involved, and their
expected applicability to the most general ``families'' of
$L$-functions, we denote
$$
\expect_k(\alpha(F))=\frac{1}{\displaystyle{\sum_{F\in
      \cusp_k^*}{\weight_k^F}}} \sum_{F\in
  \cusp_k^*}{\weight_k^F\alpha(F)}
$$
for $k\geq 2$ even and for any complex numbers $(\alpha(F))$. This is
the averaging operator for a probability measure depending on $k$, and
we know from the previous results that
$$
\expect_k(\alpha(F))\sim \sum_{F\in \cusp_{k}^*}{\weight_k^F\alpha(F)}
$$
as $k\rightarrow +\infty$ over even integers. We denote by
$\proba_k(\bullet)$ the associated probability. We also recall that
$\cusp_{2k}^{\flat}$ is the set of cusp forms which are not
Saito-Kurokawa forms, and we denote by $\cusp_{2k}^\sharp$ the
complementary set of Saito-Kurokawa lifts.

\subsection{Direct applications}\label{ssec:direct}

We start with direct consequences of the local equidistribution.  The
first is partly superseded by the proof of the generalized Ramanujan
conjecture in our case~\cite{weissram}, but it may still be taken as
an indication that the special Saito-Kurokawa modular forms which fail
to satisfy it are ``few'', even when counted with our special weights.

\begin{proposition}\label{pr:rp}
  \emph{(1)} Fix a prime $p$. Then ``most'' $F\in \cusp_k^*$ satisfy
  the generalized Ramanujan conjecture at $p$, in the sense
  that we have
$$
\proba_{2k}(\pi_p(F)\text{ is not tempered})
=\proba_{2k}(\pi_p(F)\notin X\subset X_p)
\longrightarrow 0
$$
as $k\ra +\infty$.
\par
\emph{(2)} Let $\alpha(F)$ be any bounded function defined for $F$
which are Saito-Kurokawa lifts, i.e., $F\in \cusp_k^{\sharp}$. Then
we have
$$
\lim_{k\ra +\infty}{
\expect_{2k}(\alpha(F)\mathbf{1}_{\{\text{$F\in \cusp_{2k}^{\sharp}$}\}})
}
=0.
$$
\end{proposition}

\begin{proof}
Since the limiting measure $\mu_p$ is supported on $X\subset X_p$,
this is immediate.
\end{proof}

In particular, it follows that the measure $\nu_{\sop,k}^{\flat}$
defined as $\nu_{\sop,k}$, but with $F$ restricted to
$\cusp_k^{\flat}$, also converge weakly, as $k\ra +\infty$, to
$\mu_{\sop}$.  We will denote by $\expect_k^{\flat}(\alpha(F))$ the
average
$$
\expect_k^{\flat}(\alpha(F))
=\expect_k(\alpha(F)\mathbf{1}_{\{F\in \cusp_k^{\flat}\}}),
$$
so that the previous result means that this is still, asymptotically,
a probability average.
\par
The next result has the feel of a ``Strong Approximation'' theorem:

\begin{proposition}[``Strong approximation'']\label{pr:strong}
  \emph{(1)} Let $\mathrm{Aut}$ denote the set of all cuspidal
  automorphic representations on $\GSp(4,\A)$. Then for any finite set
  of primes $\sop$, the local components $\pi_{\sop}$ for those
  $\pi\in\mathrm{Aut}$ unramified at the primes in $\sop$ form a dense
  subset of $X^{\sop}$.
\par
\emph{(2)} Fix a finite set of prime $\sop$, and let $(\eps_p)_{p\in
  \sop}$ be signs $\pm 1$. There exist infinitely many Siegel cusp
forms $F$ of level $1$ which are Hecke-eigenforms such that the Hecke
eigenvalues at all $p\in \sop$ have sign $\eps_p$.
\end{proposition}

\begin{proof}
(1) The support of the limiting measure $\mu_{\sop}$ is $X^{\sop}$, hence the
  result is again immediate (with the much more precise information
  that denseness holds already for $\pi$ associated to Siegel cusp
  forms of full level).
\par
(2) This sample application follows from the fact that the Hecke
eigenvalue at a prime $p$ is $a_p+a_p^{-1}+b_p+b_p^{-1}$ for $F\in
\cusp_k^*$, and it is clear from the formulas for the Haar measure
$\mu$~(\ref{eq-haar-measure}) and for the density function
$\Delta_p$~(\ref{eq:density}) that
$$
\mu_p(\{(a_p,b_p)_{p\in \sop}\,\mid\, \text{the sign of
  $a_p+a_p^{-1}+b_p+b_p^{-1}$ is $\eps_p$}\})>0
$$
for any prime $p$ and hence 
$$ 
\mu_{\sop}(\{(a_p,b_p)_{p\in \sop}\,\mid\, \text{the sign of
  $a_p+a_p^{-1}+b_p+b_p^{-1}$ is $\eps_p$ for $p\in \sop$}\})>0
$$
(and hence, for $k$ large enough, at least one $F\in \cusp_k^*$
satisfies those local conditions).
\end{proof}

\begin{remark} In fact, our local equidistribution shows much
  more. For instance, for any non-negligible subset of $T$ of $X^{\sop}$
  and any fundamental discriminant $-d$, our result shows that one can
  find infinitely many Siegel modular forms whose local components at
  $\sop$ lie in $T$ and the sum of whose Fourier coefficients of
  discriminant $-d$ is non-zero. Since such sums of Fourier
  coefficients of Siegel modular forms are conjecturally related to
  central critical $L$ values of the twisted forms --- see
  Section~\ref{sec:lowlying} --- this can be interpreted as a
  (conditional) result on the plentitude of Siegel modular forms with
  prescribed local behavior and non-vanishing central critical values.
\end{remark}

\subsection{Averaging $L$-functions}\label{ssec-averaging}

The spin $L$-function associated to $F\in \cusp_k^*$ is defined, in terms
of the Satake parameters $(a_p,b_p)$, by the Euler product
$$
L(F,s)=\prod_p{(1-a_pp^{-s})^{-1}(1-b_pp^{-s})^{-1}
(1-a_p^{-1}p^{-s})^{-1}(1-b_p^{-1}p^{-s})^{-1}}\ ;
$$
it is explained in~\cite{asgsch} that this is a Langlands $L$-function
associated with the Spin representation $\Spin(5,\C)\ra
\GL(4,\C)$. This Spin group is the Langlands dual group of $\Sp(4)$,
and since $\Spin(5)\simeq \Sp(4)$, this is natural in view of the
parametrization of the local representations in terms of semisimple
conjugacy classes of $\USp(4,\C)$ which is described in the
introduction. From the point of view of $\Sp(4)$, this is the
Langlands $L$-function corresponding to the natural representation
$\GSp(4)\subset \GL(4)$.
\par
The idea is that in the region of absolute convergence (which is
$\Reel(s)>1$, for $F$ not a Saito-Kurokawa lift), the average of such
a product is the average of asymptotically independent random
variables, and hence will be the product of the averages for the
limiting distributions at each $p$. Saito-Kurokawa lifts, being
asymptotically negligible, do not cause much trouble in that case.
\par
To go to the details, we first recall that, from work of
Andrianov~\cite{andrianov}, it is known that $L(F,s)$ has the basic
standard analytic properties expected from an $L$-function; it is
self-dual with root number $(-1)^k$ and unramified at finite places;
it extends to a meromorphic function of $s$, and the completed
$L$-function
$$
\Lambda(F,s)=(2\pi)^{-2s}\Gamma(s+1/2)\Gamma(s+k-3/2)L(F,s)
$$
satisfies
$$
\Lambda(F,s)=(-1)^k\Lambda(F,1-s).
$$
\par
Furthermore, if $k$ is odd, the $L$-function is entire; otherwise, it
may have poles at $s=3/2$ and $s=-1/2$, and this happens precisely
when $F$ is a Saito-Kurokawa lift (i.e., if $F\in \cusp_k^{\sharp}$,
using the notation of Theorem~\ref{th-average-l-functions}). In that
case, the $L$-function is given by
\begin{equation}\label{eq-spinor-st}
L(F,s)=\zeta(s-1/2)\zeta(s+1/2)L(f,s)
\end{equation}
for some classical cusp form $f$ of weight $2k-2$ on $\SL(2,\Z)$ (the
$L$-function of which is also normalized so that the critical line is
$\Reel(s)=1/2$). Note that although, in general, there are other
automorphic forms on $\GSp(4)$ where the $L$-function have poles, the
Saito-Kurokawa lifts are the only ones which are holomorphic with
level $1$ (we refer to Piatetski-Shapiro's paper~\cite{ps-sk} for more
details).

\begin{remark}
  For completeness, even though we do not need it here, let us recall
  the corresponding results for the other ``standard'' $L$-function,
  which is the degree-five $L$-function coming from the projection
$$
\mathrm{pr}\,:\, \Spin(5,\C)\ra \SO(5,\C)\subset \GL(5,\C).
$$
\par
This $L$-function has the form
\begin{equation}\label{eq-pr-l-function}
L(F,\mathrm{pr},s)=\prod_p{
\Bigl((1-p^{-s})
(1-a_pb_pp^{-s})(1-a_pb_p^{-1}p^{-s})
  (1-a_p^{-1}b_pp^{-s})(1-(a_pb_p)^{-1}p^{-s}\Bigr)^{-1}
}.
\end{equation}
\par
From work of Mizumoto~\cite{miz81}, it is known that
$L(F,\mathrm{pr},s)$ has the basic standard analytic properties
expected from an $L$-function; it is self-dual with root number $1$
and unramified at finite places; it extends to an \emph{entire}
function of $s$, and the completed $L$-function
$$
\Lambda(F,\mathrm{pr},s)=2^{-2s} \pi^{-5s/2}
\Gamma\left(\frac{s+1}{2}\right)\Gamma(s+k-1)\Gamma(s+k-2)L(F,s)
$$
satisfies
$$
\Lambda(F,\mathrm{pr},s)=\Lambda(F,\mathrm{pr},1-s).
$$
\end{remark}

Let us now return to the spin $L$-function. To average the
$L$-function, we express it in additive terms. For this purpose, we
denote by
$$
\pi_p\,:\, \cusp_k^*\ra X_p
$$
the map $F\mapsto \pi_p(F)$, the local component of the automorphic
representation $\pi_F$ associated with $F$ as described earlier, which
we identify with $(a_p(F),b_p(F))\in Y_p$.
\par
Expanding the Euler factors in powers of $p^{-s}$ and then expanding
the product into Dirichlet series, we find the expression
$$
L(F,s)=\sum_{n\geq 1}{\lambda(F,n)n^{-s}}
$$
in the region of absolute convergence, where
$$
\lambda(F,n)=\prod_{p\mid n}{H_{v_p(n)}(\pi_p(F))},\quad\text{for}
\quad
n=\prod_{p\mid n}{p^{v_p(n)}},
$$
in terms of functions $H_m$, $m\geq 0$, on $Y_p$ given by the
symmetric functions
$$
H_m(a,b)=\sum_{k_1+k_2+k_3+k_4=m}{a^{k_1-k_3}b^{k_2-k_4}}
$$
(note that $H_m$ is independent of $p$, though that is not crucial in
what follows, and that it is well-defined on $X_p$ since it is
invariant under the Weyl group).
\par
If $\Reel(s)>1$, the series $L(F,s)$ converge absolutely at $s$ for
$F\in\cusp_{2k}^{\flat}$, and we have
$$
\expect_{2k}^{\flat}(L(F,s))=\sum_{n\geq 1}{\expect_{2k}^{\flat}
(\lambda(F,n))n^{-s}}.
$$
\par
Fix $n$ first, and factor it as before
$$
n=\prod_{p\mid n}{p^{v_p}}.
$$
\par
By our local equidistribution theorem applied to $\nu_{\sop,2k}^{\flat}$,
we have
$$
\expect^{\flat}_{2k}(\lambda(F,n))= \expect^{\flat}_{2k}
\Bigl(\prod_{p\mid  n}{H_{v_p}(\pi_p(F))}\Bigr) \ra
\prod_{p\mid  n}{\int_{X_p}{H_{v_p}(x)d\mu_p(x)}}
$$
as $k\ra +\infty$. Therefore, by the dominated convergence theorem, we
have
$$
\sum_{n\geq 1}{\expect^{\flat}_{2k}(\lambda(F,n))n^{-s}}
\ra \sum_{n\geq 1}{
\Bigl(\prod_{p\mid n}{\int_{X_p}{H_{v_p}(x)d\mu_p(x)}}\Bigr)
n^{-s}}
$$
since, using the formula
$$
\frac{1}{6}(m+1)(m+2)(m+3)
$$
for the number of monomials of degree $m$ in $4$ variables (the number
of terms in $H_m$) we have
$$
\Bigl|\expect^{\flat}_{2k}(\lambda(F,n))n^{-s}\Bigr| \leq
n^{-\sigma}\prod_{p\mid n}{(v_p+3)^3}
$$
for all $n\geq 1$ and ${k}$, which defines an absolutely convergent
series for $\sigma=\Reel(s)>1$. (We use here the generalized Ramanujan
conjecture, proved in this case by Weissauer~\cite{weissram}.)
\par
Now we refold back the limiting expression as an Euler product:
\begin{align*}
\sum_{n\geq 1}{
\Bigl(\prod_{p\mid n}{\int_{X_p}{H_{v_p}(x)d\mu_p(x)}}\Bigr)
n^{-s}}&=
\prod_p{\sum_{l\geq 0}{
p^{-{l}s}\int_{X_p}{H_{l}(x)d\mu_p(x)}
}}\\
&=
\prod_p{
\int_{X_p}{
L_p(x,s)d\mu_p(x)
}
},
\end{align*}
where $L_p(x,s)$ is the local $L$-factor of a local representation
$x=(a,b)\in X_p$ defined in the lemma above (the Euler expansion is
justified again by the fact that the series on the left is absolutely
convergent, as we checked in the lemma). Thus we have proved
\begin{equation}\label{e:average-flat}
\lim_{k\ra +\infty}{\expect_{2k}^{\flat}(L(F,s))}= \prod_p{ \int_{X_p}{
    L_p(x,s)d\mu_p(x) } }.
\end{equation}
\par
Now assume $\Reel(s)>1$ and $s\not=3/2$. Then all spin $L$-functions
of Saito-Kurokawa lifts are well-defined at $s$, and we therefore also
want to have average formulas involving them. If $\Reel(s)>3/2$, this
is immediate by the previous argument. Otherwise, we have
$$
\sum_{F\in \cusp_{2k}^{\sharp}}{\weight_{2k}^FL(F,s)}=
\zeta(s-1/2)\zeta(s+1/2)
\sum_{F\in \cusp_{2k}^{\sharp}}{\weight_{2k}^FL(f_F,s)},
$$
where $f_F$ is a classical modular form from which $F$ arises. The
$L$-function $L(f_F,s)$ is now absolutely convergent, and its values
are bounded for all Saito-Kurokawa lifts (by the generalized Ramanujan
conjecture, for instance). Thus we have
$$
\lim_{k\ra +\infty}\sum_{F\in \cusp_{2k}^{\sharp}}{\weight_{2k}^FL(F,s)}=0
$$
by Proposition~\ref{pr:rp}, (2), and this combined
with~(\ref{e:average-flat}) gives the result
$$
\lim_{k\ra +\infty}{\expect_{2k}(L(F,s))}= \prod_p{ \int_{X_p}{
    L_p(x,s)d\mu_p(x) } }.
$$
\par
At this point, it is clear how to extend this to other Langlands
$L$-functions. Indeed, let
$$
\rho\,:\, \GSp(4,\C)\ra \GL(r,\C)
$$ 
be an algebraic representation. The Langlands $L$-function is defined
by 
$$
L(F,\rho,s)=L(\pi_F,\rho,s)=\prod_{p}{\det(1-\rho(x_p(F))p^{-s})^{-1}}
$$
where
$$
x_p(F)=x_p(a_p,b_p)=\begin{pmatrix}
a_p& &&\\
&b_p&&\\
&&a_p^{-1}&\\
&&&b_p^{-1}
\end{pmatrix}
$$
is the semisimple conjugacy class of $\GSp(4,\C)$ associated with
$\pi_p(F)$. We can expand
$$
\det(1-\rho(x_p(F))p^{-s})^{-1}
=\prod_{1\leq j\leq r}{(1-\alpha_j(a_p,b_p)p^{-s})^{-1}}
$$
for some polynomial functions $\alpha_j$ on $X_p$, and then we can
repeat the same argument to derive~(\ref{e:general-rho}) with 
$$
\lim_{k\ra +\infty}{\expect_{2k}(L(F,\rho,s))}
=\prod_{p}{
\int_{X}{
\det(1-\rho(x_p(a,b))p^{-s})^{-1}d\mu_p(a,b)}},
$$
when $s$ is in the region of common absolute convergence of all $F$.
\par
Finally, to get the precise expression in
Theorem~\ref{th-average-l-functions} for the spin $L$-function, we
note that the special case~(\ref{e:onevar-sugano}) of Sugano's formula
(Theorem~\ref{t:sugano}), with $Y=p^{-s}$, gives the explicit
decomposition
$$
L_p(x,s)=
\Bigl(1 -\lambda_pp^{-1/2-s} +
\left(\frac{\field}{p}\right)p^{-1-2s}\Bigr)^{-1}
\Bigl(\sum_{l \ge 0} U^{l, 0}_p(a,b)p^{-ls}\Bigr)
$$
for any prime $p$. Applying Proposition~\ref{propplancherelmeasure},
we get therefore the simple expression
$$
\int_{X_p}{L_p(x,s)d\mu_p(x)}=\frac{1}{
1 -\lambda_pp^{-1/2-s} +
\left(\frac{\field}{p}\right)p^{-1-2s}},
$$
and (using the definition of $\lambda_p$) we recognize that this is
$$
L(\Lambda,s+1/2),
$$
where $\Lambda$ is the class group character of $\field=\Q(\sqrt{-d})$
defining our fixed Bessel models. When $d=4$ and $\Lambda$ is trivial,
this is $\zeta(s+1/2)L(\chi_4,s+1/2)$, which is the
formula~(\ref{e:firstaverage}). 

\begin{remark}
  In fact, this second argument for the spin $L$-function can be used
  to bypass the first one (which therefore requires only that we work
  with the family of functions $U_p^{l,0}(a,b)$).
\end{remark}

\begin{remark}
  Note that although Theorem~\ref{th-average-l-functions} was stated
  in the introduction only for averages with respect to the weight
  $\weight_{k}^F := \weight_{k, d, \Lambda}^F$ in the special case
  $d=4$, $\Lambda=1$, our proof actually works for general $d$ and
  $\Lambda$.
\end{remark}

The proof also gives the following fact concerning the limit
averages:

\begin{lemma}
  For $p$ prime, let $\mu_p$ be the limiting measure in the local
  equidistribution result and let
$$
L_p(x,s)=\prod_p{(1-a p^{-s})^{-1}
(1-b p^{-s})^{-1}
(1-a^{-1} p^{-s})^{-1}
(1-b^{-1} p^{-s})^{-1}}
$$
be the local $L$-function for $x\in X_p$. Then the Euler product
$$
\prod_p{\int_{X_p}{L_p(x,s)d\mu_p(x)}}
$$
is absolutely convergent for $\Reel(s)>1/2$.
\end{lemma}

\begin{proof}
According to what we have said, we have
$$
\int_{Y_p}{L_p(x,s)d\mu_p(x)}=
\frac{1}{1 -\lambda_pp^{-1/2-s} +
\left(\frac{\field}{p}\right)p^{-1-2s}}
$$
and the result is then obvious.
\end{proof}


\subsection{Weights and averages over Saito-Kurokawa lifts}
\label{sec:sk}

In this section, we will explicitly compute $\weight_{k}^F :=
\weight_{k, d, \Lambda}^F$ when $F$ is a Saito-Kurokawa lift. This
will lead to a stronger version of the second part of
Proposition~\ref{pr:rp}. This simple fact is included because it may
be helpful for further investigations.
\par
Let $k>2$ be even and let $\classical_{2k-2}^*$ denote the Hecke basis
of the space of holomorphic cusp forms on $\GL(2)$ of weight $2k-2$
and full level. Let $F \in \cusp_{k}^*$ be the (unique) Saito-Kurokawa
lift of $f \in \classical_{2k-2}^*$, so that the spinor $L$-function
is given by~(\ref{eq-spinor-st}). As usual, we let
$$ 
F(Z)=\sum_{T>0}a(F,T)e(\Tr (TZ))
$$ 
be the Fourier expansion of $F$. It is well-known (see~\cite{eichzag}
for instance) that $a(F,T)$ then depends only on the determinant of
$T$. In particular, it follows that $\weight_{k, d, \Lambda}^F = 0$
whenever $\Lambda \neq 1$. So, we assume that $\Lambda = 1$ and
shorten $\weight_{k, d, 1}^F$ to $\weight_k^F$.
\par
Let
$$
\tilde{f} (z) = \sum_{n>0}c(n) e(n z)
$$ 
be the cusp form of half-integer weight $k - \frac{1}{2}$ on
$\Gamma_0(4)$ that is associated to $f$ via the Shimura
correspondence. Then, by~\cite[Th. 6.2, equation (6)]{eichzag}, we have
\begin{equation}\label{sk1}
a(T) = c(d)
\end{equation}
for any positive-definite semi-integral matrix $T$ of determinant
$d/4$. On the other hand, by~\cite{brown07}, we have
\begin{equation}\label{sk2}
  \langle F, F \rangle = \frac{k-1}{2^4 \cdot 3^2 \cdot \pi} \cdot
  \frac{|c(d)|^2}{d^{k-\frac{3}{2}}} \cdot \frac{L(f, 1)}{L(f\times \chi_d,
    \frac{1}{2})} \langle f, f \rangle,
\end{equation}
whenever $c(d)$ is non-zero.

Using~\eqref{sk1},~\eqref{sk2} and the definition of $\weight_k^F$, it
follows that 
$$
\weight_k^F = \frac{(48\pi)^2 h(-d)}{w(-d)(k-1)(k-2)}
\frac{ \Gamma(2k - 3)}{(4 \pi)^{2k-3} \langle f, f \rangle}
\frac{L(f \times \chi_d, \frac{1}{2})}{L(f,1)}.
$$
\par
Now we consider the average
$$
\sum_{f \in \classical_{2k-2}^*}
\frac{\Gamma(2k-3)}{(4\pi)^{2k-3}\langle f,f\rangle} \frac{L(f \times
  \chi_d, \frac{1}{2})}{L(f,1)}.
$$
\par
It is very likely that one can prove an asymptotic formula for this
quantity as $k\ra +\infty$ (possibly using the methods of Ramakrishnan
and Rogawski in~\cite{rogram}). However, to deal with it quickly, we
observe first that $L(f \times \chi_d, \frac{1}{2})$ and $L(f,1)$ are
both non-negative (e.g., because $L(f,s)$ is real-valued, has no zero
with $\Reel(s)>1$ and tends to $1$ as $s\ra +\infty$, and the ratio is
non-negative by the above). Then, using the fact that $L(f,s)$ has no
Siegel zeros (a result of Hoffstein and
Ramakrishnan~\cite{hoffstein-ramakrishnan}), one gets in the usual way
a lower bound
$$
L(f,1)\gg \frac{1}{\log k},
$$
and therefore 
$$
\sum_{f \in \classical_{2k-2}^*} 
\frac{\Gamma(2k-3)}{(4\pi)^{2k-3}\langle f,f\rangle}
\frac{L(f \times \chi_d, \frac{1}{2})}{L(f,1)}
\ll
(\log k)\sum_{f \in \classical_{2k-2}^*} 
\frac{\Gamma(2k-3)}{(4\pi)^{2k-3}\langle f,f\rangle}
L(f \times \chi_d, \frac{1}{2}).
$$
\par
Next, from the results of Duke~\cite{duke}, one gets
$$
\sum_{f \in \classical_{2k-2}^*}
\frac{\Gamma(2k-3)}{(4\pi)^{2k-3}\langle f,f\rangle} L(f \times
\chi_d, \frac{1}{2})\ll 1
$$
for $k\geq 2$, where the implied constant depends on $d$. The
following Proposition, which strengthens Proposition~\ref{pr:rp}, (2),
is then an immediate consequence:

\begin{proposition} 
  Suppose $\alpha(F)$ is a complex valued function defined for
  Saito-Kurokawa lifts and satisfying for some $\delta>0$ the
  inequality 
$$
\alpha(F) \ll k^{2- \delta}
$$ 
for $F\in \cusp_{2k}^{\sharp}$. Then
$$
\lim_{k\ra +\infty}{ \expect_k(\alpha(F)\mathbf{1}_{\{\text{$F\in
      \cusp_k^{\sharp}$}\}}) } =0.
$$
\end{proposition}

Using weak bounds, like the convexity bound, this applies for instance
to $\alpha(F)=L(F,1/2+it)$ for fixed $t\not=0$.

\subsection{Low-lying zeros, Katz-Sarnak symmetry type and
  B\"ocherer's conjecture}
\label{sec:lowlying}

The determination of the distribution of low-lying zeros of the spin
$L$-functions (assuming the Generalized Riemann Hypothesis) for
restricted test functions is not difficult once a quantitative
equidistribution statement is known. Conjecturally, the answer
indicates which ``symmetry type'' (in the sense of Katz-Sarnak) arises
for the family. However, we will see that the answer in our case is
surprising, and gives some global evidence towards a well known
conjecture of B\"ocherer.

We now prove Theorem~\ref{th:low-lying}. This type of computation is
quite standard by now, and is known to succeed as soon as
``approximate orthogonality'' has been proved with a power saving with
respect to the analytic conductor, which is the case thanks to our
quantitative equidistribution theorem (precisely,
from~(\ref{eq:ortho-quant})). We may therefore be brief, as far as
technical details are concerned (we refer to,
e.g.,~\cite{duenez-miller}, where families derived from classical
$\GL(2)$ cusp forms are treated with respect to the weight). However,
since the main term arising from this computation has some meaning, we
must justify it carefully.

As before, note that although Theorem~\ref{th:low-lying} was stated in
the introduction only for averages with respect to the weight
$\weight_{k}^F := \weight_{k, d, \Lambda}^F$ in the special case
$d=4$, $\Lambda=1$, we actually prove it for any $d$ (we stick to
$\Lambda = 1$).

\begin{proof}[Proof of Theorem~\ref{th:low-lying}]
  Throughout this proof, $\weight_{k}^F$ denotes $\weight_{k,d,1}^F$.
  For given $F$, we write
$$
-\frac{L'}{L}(F,s)=\sum_{n\geq 1}{c(F,n)\Lambda(n)n^{-s}}
$$
the logarithmic derivative of the spinor $L$-function which is
supported on powers of primes and where $\Lambda(n)$ is the von
Mangoldt function and, for $n=p^m$, $m\geq 1$, we have
$$
c(F,p^m)=a_p^m+a_p^{-m}+b_p^m+b_p^{-m}=
\Tr(\pi_p(F)^m),
$$
where $\pi_p(F)$ is interpreted as a conjugacy class in $\USp(4,\C)$.
\par
We can apply the following form of the ``explicit formula'' (see,
e.g.,~\cite[Th. 5.12]{ant}) to relate sums over zeros to sums over
primes involving those coefficients: denoting
$$
\gamma(s)=(2\pi)^{-2s}\Gamma(s+1/2)\Gamma(s+k-3/2)
$$
the common gamma factor for all $L(F,s)$, for any test-function $\psi$
which is even and of Schwartz class on $\R$, we have
$$
\sum_{\rho}{\psi\Bigl(\frac{\gamma}{2\pi}\Bigr)}=
\frac{1}{2\pi}\int_{\R}{\Bigl(\frac{\gamma'}{\gamma}(1/2+it)+
  \frac{\gamma'}{\gamma}(1/2-it)\Bigr)\psi(x)dx}
-2\sum_{n}{\hat{\psi}(\log n)\frac{c(F,n)\Lambda(n)}{\sqrt{n}}}.
$$
\par
We apply this to
$$
\psi(x)=\varphi\Bigl(\frac{x}{2\pi}\log(k^2)\Bigr),\quad\quad
\hat{\psi}(t)=\frac{\pi}{\log k}
\hat{\varphi}\Bigl(\frac{\pi t}{\log k}\Bigr)
$$
where $\varphi$ is an even Schwartz function with Fourier transform
supported in $[-\alpha,\alpha]$. After treating the gamma factor using
the formula
$$
\frac{\Gamma'}{\Gamma}(k-1+it)+\frac{\Gamma'}{\Gamma}(k-2-it)=
2\log k+O(t^2 k^{-2}),
$$
which follows from Stirling's formula (see, e.g.,~\cite[\S 3.1.1, \S
3.1.2]{duenez-miller} for precise details of these computations) and
spelling out the von Mangoldt function, we obtain
\begin{equation}
D_{\varphi}(F)=\int_{\R}{\varphi(x)dx}
-\frac{2}{\log(k^2)}\sum_{m\geq 1}{
\sum_{p}{\frac{\log p}{p^{m/2}}
c(F,p^m)
\hat{\varphi}\Bigl(m\frac{\log p}{\log(k^2)}\Bigr)
}
}
+O((\log k)^{-1}).
\end{equation}
\par
Averaging over $F$ leads to
$$
\expect_k(D_{\varphi}(F))
=\hat{\varphi}(0)
-\frac{2}{\log(k^2)}\sum_{m\geq 1}{
\sum_{p}{\frac{\log p}{p^{m/2}}
\expect_k(c(F,p^m))
\hat{\varphi}\Bigl(m\frac{\log p}{\log(k^2)}\Bigr)
}
}
+O((\log k)^{-1}).
$$
\par
As usual, easy estimates give
$$
\lim_{k\ra +\infty}
\frac{1}{\log(k^2)}
\sum_{m\geq 3}{
\sum_{p}{\frac{\log p}{p^{m/2}}
\expect_k(c(F,p^m))
\hat{\varphi}\Bigl(m\frac{\log p}{\log(k^2)}\Bigr)
}
}=0
$$
(the series over primes being convergent even without the
compactly-supported test function).
\par
In the term $m=1$, we have
\begin{align}
\expect_k(c(F,p))&= \expect_k(a_p+b_p+a_p^{-1}+b_p^{-1})
\nonumber\\
&=\expect_k(U_p^{1,0}(\pi_p(F)))+\frac{\lambda_p}{\sqrt{p}}
=\frac{\lambda_p}{\sqrt{p}}+O(p^{1+\eps}k^{-2/3})
\label{eq-m1}
\end{align}
by~(\ref{eq-u10}) and~(\ref{eq:ortho-quant}), and hence the
contribution of $m=1$, which is given by
$$
\frac{2}{\log(k^2)}\sum_{p}{\frac{\log p}{p^{1/2}}
\expect_k(c(F,p))
\hat{\varphi}\Bigl(\frac{\log p}{\log(k^2)}\Bigr)
}
$$
is equal to
$$
\frac{2}{\log(k^2)}
\sum_{p}{\frac{\lambda_p\log p}{p}
\hat{\varphi}\Bigl(\frac{\log p}{\log(k^2)}\Bigr)}
+O\Bigl(\frac{1}{k^{2/3}\log k}\sum_{p\leq k^{2\alpha}}
{p^{1/2+\eps}}\Bigr)
=M_k(\varphi)+O(k^{5\alpha/2-2/3+\eps}).
$$
for any $\eps>0$, where
\begin{align*}
M_k(\varphi)&=\frac{2}{\log(k^2)}
\sum_{p}{\frac{\lambda_p\log p}{p}
\hat{\varphi}\Bigl(\frac{\log p}{\log(k^2)}\Bigr)}
\\
&=2\int_{[1,+\infty[}{
\hat{\varphi}\Bigl(\frac{\log y}{\log(k^2)}\Bigr)
\frac{1}{\log(k^2)}\frac{dy}{y}+o(1)
}\\
&=2\int_{[0,+\infty[}{\hat{\varphi}(x)dx}=\varphi(0)+o(1),
\end{align*}
(since $\varphi$ is even), by summation by parts using the Prime
Number Theorem and the fact that $\lambda_p=2$ or $0$ for primes with
asymptotic density $1/2$ each, so the average value is $1$ (see,
e.g.,~\cite[Lemma 2.7]{duenez-miller}).
\par
Now we consider the term $m=2$. Although we could appeal to the
general estimate~(\ref{eq:quant-local}), we will use an explicit
decomposition and~(\ref{eq:ortho-quant}). First,
using~(\ref{e:onevar-sugano}), we have
\begin{equation}\label{eq-u20}
U^{2,0}_p(\pi_p(F))=1-\frac{\lambda_p}{\sqrt{p}}U^{1,0}_p(\pi_p(F))
+c(F,p^2)+\tau(\pi_p(F))+O(p^{-1})
\end{equation}
where the implied constant is absolute and 
$$
\tau(a,b)=1+ab+ab^{-1}+a^{-1}b+(ab)^{-1}
$$ 
as in Theorem~\ref{t:sugano}. By~(\ref{eq-u01}), we have
$$
\tau(\pi_p(F))(1+\alpha_p)=U^{0,1}_p(\pi_p(F))+
\beta_p\sigma(\pi_p(F))+O(p^{-1})
$$
where the quantities $\alpha_p\ll p^{-1}$, $\beta_p\ll p^{-1/2}$ do
not depend on $F$, and the implied constants are absolute. Averaging
(with $\expect_k(\cdot)$) from this last formula and
using~(\ref{eq-m1}), we find
$$
\expect_k(\tau(\pi_p(F)))\ll p^{-1}+p^{3/2+\eps}k^{-2/3},
$$
and from~(\ref{eq-u20}), we therefore derive
$$
\expect_k(c(F,p^2))=-1+O(p^{-1}+p^{2+\eps}k^{-2/3})
$$
for any $\eps>0$. Consequently, we see that the term $m=2$, after
averaging, is given by
\begin{align*}
\frac{2}{\log(k^2)}\sum_{p}{\frac{\log p}{p}
\expect_k(c(F,p^2))
\hat{\varphi}\Bigl(2\frac{\log p}{\log(k^2)}\Bigr)
}&=
-\frac{2}{\log(k^2)}\sum_{p}{
\frac{\log p}{p}\hat{\varphi}\Bigl(2\frac{\log p}{\log(k^2)}\Bigr)
}+\\
&+O\Bigl(
\frac{1}{k^{2/3}\log k}\sum_{p\leq k^{\alpha}}{p^{1+\eps}
}+(\log k)^{-1}\Bigr)\\
&=-N_k(\varphi)+O((\log k)^{-1}+k^{2\alpha-2/3+\eps}),
\end{align*}
where
$$
N_k(\varphi)=\frac{2}{\log(k^2)}\sum_{p}{ \frac{\log
    p}{p}\hat{\varphi}\Bigl(2\frac{\log p}{\log(k^2)}\Bigr)
}=\frac{\varphi(0)}{2}+o(1),
$$
by computations similar to that of $M(\varphi)$ before.
\par
We notice that the contribution of the main terms for $m=1$
and $m=2$ together are
$$
-\varphi(0)+\frac{\varphi(0)}{2}=-\frac{\varphi(0)}{2},
$$
which gives a main term
$$
\hat{\varphi}(0)-\frac{\varphi(0)}{2}=\int_{\R}{\varphi(x)d\sigma_{Sp}(x)}.
$$
\par
Moreover, the error terms in both are negligible
as long as $5\alpha/2-2/3<0$, i.e., $\alpha<4/15$. Under this
condition, we obtain therefore
$$
\expect_k(D_{\varphi}(F))=
\int_{\R}{\varphi(x)dx}+o(1),
$$
as $k\ra +\infty$, which is the desired conclusion.
\end{proof}

\begin{remark}
  It is interesting to note that, to understand the logarithmic
  derivative of the spin $L$-function $L(F,s)$, one needs to involve
  the average of the quantity
$$
\tau(\pi_p(F))=1+a_pb_p+a_pb_p^{-1}+a_p^{-1}b_p+(a_pb_p)^{-1},
$$
which is the coefficient of $p^{-s}$ in the \emph{projection}
$L$-function $L(F,\mathrm{pr},s)$ (see~(\ref{eq-pr-l-function})). This
illustrates the fact that, in the study of automorphic forms on
groups of rank $r\geq 2$, all Langlands $L$-functions (or, at least,
those associated with the $r$ fundamental representations of the dual
group) are intimately linked, and must be considered together.
\end{remark}

We now comment on the relation of Theorem~\ref{th:low-lying} with
B\"ocherer's conjecture. In the literature, a density of low-lying
zeros given by the measure $d\sigma_{Sp}$ (as we proved is the case)
is taken as an indication of \emph{symplectic symmetry type} (the
basic example being the family of real Dirichlet
characters). Intuitively, these are families of central $L$-values of
self-dual $L$-functions with sign of functional equation $+1$ for
which the central point of evaluation has no special meaning.
However, although our families are indeed self-dual, a symplectic
symmetry seems very unlikely for our family, for at least two reasons:
first, $1/2$ is a critical point in the sense of Deligne, and second,
the forms of odd weight have functional equations with sign $-1$. 
\par
There is a natural explanation for the discrepancy: the Fourier
coefficient $|a(d,1;F)|^2$ appearing in the weight
$$
\weight_F^{k}=c_{k,d} \frac{|a(d,1;F)|^2}{\langle F, F\rangle}.
$$
involves probably deeper arithmetic content than one might naively
think. Indeed, in~\cite{boch-conj}, B\"ocherer made the following
remarkable conjecture:

\begin{conjecture}[B\"ocherer's Conjecture]
  For any $F \in \cusp_{2k}^*$, there exists a constant $C_F$
  depending only on $F$ such that for all fundamental discriminants
  $-d <0$ one has $$L(F \times \chi_d , \frac{1}{2}) = C_F\cdot
  d^{1-2k}w(-d)^{-2}\cdot |a(d,1;F)|^2,$$ where $\chi_d$ denotes the
  quadratic character associated to the extension $\Q(\sqrt{-d})$.
\end{conjecture}
\par
B\"ocherer proved this conjecture for Eisenstein series and
Saito-Kurokawa lifts in~\cite{boch-conj}. Later, he and Schulze-Pillot
proved an analogue of this conjecture (for Siegel modular forms with
level) in the case of Yoshida lifts. More recently, works of
Furusawa-Shalika~\cite{furusawa-shalika-fund},
Furusawa-Martin~\cite{furusawa-martin} and
Furusawa-Martin-Shalika~\cite{furmarsha} have tried to tackle this
problem using the relative trace formula.
\par
B\"ocherer did not make any speculation about the value of the
quantity $C_F$. However recent works such as~\cite{furusawa-martin}
give some inkling of what to expect.
\par
We now show that a certain assumption on $C_F$ explains our result on
low-lying zeroes. To be more definite, we make the following
hypothesis.
\par
\begin{hypothesis}
  For non-Saito-Kurokawa forms $F\in\cusp_{2k}^*$, we have
\begin{equation}\label{eq:bocherer}
\weight_F^{2k}=L(F,\frac{1}{2})L(F\times
\chi_d,\frac{1}{2})L(\chi_d, 1)^{-1}\gamma(F)
\end{equation}
in terms of spinor $L$-functions, where $\gamma(F)>0$ is
``well-behaved'', in particular 
$$
\sum_{F \in \cusp_{2k}^*}\gamma(F)
$$
has a positive limiting average value as $k\ra +\infty$, and
$\gamma(F)$ is asymptotically independent of the central special
$L$-values.
\end{hypothesis}

In terms of Fourier coefficients, this hypothesis is equivalent to the
following specific variant of B\"ocherer's conjecture: for all
$F\in\cusp_{2k}^*$ that is not a Saito-Kurokawa lift, we should have
\begin{equation}\label{eq:bocherer2}
  L(F,\frac{1}{2})L(F\times
  \chi_d,\frac{1}{2}) = 4 \pi c_{2k}
  \gamma(F)^{-1}(d/4)^{1-2k}w(-d)^{-2} \frac{|a(d,1;F)|^2}{\langle F,
    F \rangle}.
\end{equation} 

\begin{remark}Such a formulation (involving two central values, or in
  other words a central value for the base-change of $F$ to the
  quadratic field $\Q(\sqrt{-d})$) is strongly suggested
  by~\cite[(1.4)]{furusawa-martin} and~\cite{prasadbighash}. It is
  also compatible with a conjecture of Prasad and
  Takloo--Bighash~\cite{prasadbighash}, which itself is an analogue
  (for the case of Bessel periods) of the remarkable Ichino-Ikeda
  conjecture~\cite{ichino-ikeda} dealing with $(SO(n), SO(n-1))$
  periods. In this context, it is also worth mentioning that the
  question of vanishing of the Bessel period, i.e., the vanishing of
  $a(d,\Lambda;F)$, is closely tied with the Gross-Prasad conjecture
  for $(SO(5), SO(2))$.
\end{remark}
\par
Under our stated hypothesis~\eqref{eq:bocherer}, we consider the
crucial average
$$
\sum_{F\in\cusp_{2k}^*}{
\omega_F^{2k}c(F,p)
}
$$
for a fixed prime $p$. Our goal is to show that this allows us to
recover naturally the formula~(\ref{eq-m1}) from which the
``mock-symplectic'' symmetry-type arose in the proof of
Theorem~\ref{th:low-lying} (the contribution of $p^2$ was consistent
with the expected orthogonal symmetry type). Thus,
assuming~(\ref{eq:bocherer}), we need to compute the average
$$
\sum_{F\in\cusp_{2k}^{\sharp}}{
\gamma(F)L(F,\frac{1}{2})L(F\times
\chi_d,\frac{1}{2})L(\chi_d, 1)^{-1}c(F,p)},
$$
\par
Since $c(F,p)=\lambda(F,p)$ is also the $p$-th coefficient of the
Dirichlet series $L(F,s)$, and since the analytic conductor of both
$L$-functions is about $k^2$, we see by applying a suitable
Approximate Functional Equation (and recalling that the sign of the
functional equation is $1$ for both $L$-functions) that this is
roughly
$$
2L(\chi_d, 1)^{-1}\sum_{m,n\leq k}{
\frac{\chi_d(n)}{\sqrt{mn}}\sum_{F\in\cusp_{2k}^{\sharp}}{
\gamma(F)\lambda(F,m)\lambda(F,n)\lambda(F,p)}}
$$
(the sums should involve a smooth cutoff).
\par
Under the (reasonable) assumption that the coefficients of the
Dirichlet series are asymptotically orthogonal under this
average,\footnote{\ This depends on the hypothesis that $\gamma(F)$ is
  innocuous.}  one is led to the guess that the terms which contribute
are those with $m=np$ or $n=mp$, and thus one should have
$$
\sum_{F\in\cusp_{2k}^{\sharp}}{ \gamma(F)L(F,\frac{1}{2})L(F\times
  \chi_d,\frac{1}{2})L(\chi_d, 1)^{-1}c(F,p)} \approx
\frac{1+\chi_d(p)}{\sqrt{p}}=\frac{\lambda_p}{\sqrt{p}},
$$
as $k\ra +\infty$, where the $L(1,\chi_d)$ has cancelled out with
$$
\sum_m{\frac{\chi_d(m)}{\sqrt{m}}}\approx L(1,\chi_d)
$$
(again with a smooth cutoff).
\par
But this is exactly what we proved in~(\ref{eq-m1}), and what led to
Theorem~\ref{th:low-lying}. We therefore interpret this as a (global,
averaged) confirmation of B\"ocherer's conjecture in the
form~(\ref{eq:bocherer2}).


\section*{Appendix: comparison with $\GL(2)$-families}

This section is intended to summarize some basic facts about
holomorphic Siegel modular forms and their ad\'elic counterparts, by
comparison with the case of classical modular forms for congruence
subgroups of $\SL(2,\Z)$. We also give references for the
$\SL(2)$-analogues of the results in this paper.
\par
\begin{itemize}
\item The closest analogue of our family of cusp forms is the set
  $\classical_k^*$ of primitive holomorphic cusp forms of weight $k$
  for $\SL(2,\Z)$, with trivial nebentypus, counted according to the
  weight given by
$$
\weight_f=\frac{\Gamma(k-1)}{(4\pi)^{k-1}}\frac{1}{\langle f,f\rangle}.
$$
\par
In contrast with $\cusp_k^*$, this set is the unique Hecke-eigenbasis
of the space $\classical_k$ of cusp forms of weight $k$ and level $1$;
in our context, the corresponding multiplicity one theorem is not
known (because the Fourier coefficients are not functions of the Hecke
eigenvalues), and so the Hecke basis $\cusp_k^*$ is not necessarily
unique in $\cusp_k$.
\par
Another obvious distinction is the presence of the Fourier coefficient
$a(F,1)$ in~(\ref{e:weight}). As we saw, this has crucial arithmetic
content. A way to insert this aspect ``by hand'' into the classical
case is to consider the twisted weights
$$
\weight'_f=\alpha \weight_f L(f,1/2),\quad\text{ or }
\quad
\weight'_f=\alpha \weight_f L(f,1/2)L(f\times \chi_d,1/2),
$$
where $\alpha>0$ is a constant such that
$$
\sum_{f\in \classical_k}{\weight'_f}\ra 1,
$$
as $k\ra +\infty$. (The existence of the limit that makes this
normalization possible is essentially already in Duke's
paper~\cite[Prop. 2]{duke}, where the limit is with respect to the
level.)

\item The local equidistribution theorem for $\classical_k^*$, as $k\ra
  +\infty$, is the following: for any prime $p$, the map
sending $f$ to the $p$-component of the associated automorphic
representation of $\GL(2,\A)$ can be identified with the map
$$
p\mapsto \lambda_f(p)\in\R
$$
where $p^{\frac{k-1}{2}}\lambda_f(p)$ is the $p$-th Hecke eigenvalue,
or equivalently the $p$-th Fourier coefficient. By Hecke's bound, we
have $\lambda_f(p)\in [-2\sqrt{p},2\sqrt{p}]$, and the associated
representation of $\GL(2,\Q_p)$ is the unramified principal series
obtained from the unramified characters $\alpha$, $\beta$ of
$\Q_p^{\times}$ such that
$$
\alpha(p)\beta(p)=1,\quad\quad \alpha(p)+\beta(p)=\lambda_f(p).
$$
\par
Then, for any finite set of primes $\sop$, the measures $m_{\sop,k}$ defined
as the sum of Dirac measures at $\lambda_f(p)$ for $p\in
\classical_k^*$ converge weakly to the measure
$$
n_{\sop}=\prod_{p\in \sop}{\mu_{ST}}
$$
where $\mu_{ST}$ is the Sato-Tate measure, supported  on $[-2,2]$,
given there by
$$
\frac{2}{\pi}\sqrt{1-x^2/4}dx.
$$
\item The above fact is quite easy to prove. First, the Hecke relations
  describe $\lambda_f(L)$ in terms of the factorization of $L\geq 1$,
  namely
$$
\lambda_f(L)=\prod_{p\mid L}{U_{l_p}(\lambda_f(p))},
$$
where $U_l$ is the $l$-th Chebychev polynomial defined by
$$
U_l(2\cos\theta)=\frac{\sin((l+1)\theta)}{\sin\theta}.
$$
\par
These form a basis of polynomials in one variable, and hence span a
dense subset of $C([-2\sqrt{p},2\sqrt{p}])$, with
$$
\int{U_l(x)d\mu_{ST}(x)}=\delta(l,1).
$$
\item The second ingredient is the Petersson formula; indeed, for any
  $l_p\geq 0$, let
$$
L=\prod_{p\in \sop}{p^{l_p}},
$$
and then we have
\begin{align*}
\sum_{f\in\classical_k^*}{\weight_f \prod_{p\in \sop}{
U_{l_p}(\lambda_f(p))
}}&=
\sum_{f\in\classical_k^*}{\weight_f \lambda_f(L)}\\
&=\delta(L,1)-2\pi i^{-k} \sum_{c\geq 1}{c^{-1}
S(L,1;c)J_{k-1}\Bigl(\frac{4\pi \sqrt{L}}{c}\Bigr)}
\longrightarrow \delta(L,1)
\end{align*}
as $k\ra +\infty$, where $S(L,1;c)$ denotes the standard Kloosterman
sum.  This gives the local equidistribution statement. Note that, in
contrast with our results, the limiting measure at $p$ is independent
of $p$ in this case. 
\item The Hecke relations are analogues of Sugano's formula
  (Theorem~\ref{t:sugano}); the Petersson formula and the related
  orthogonality are the analogues of
  Proposition~\ref{pr:petersson}. On the other hand, the necessary
  work to go from Fourier coefficients (controlled by Poincar\'e
  series) to Hecke eigenvalues is completely absent from the classical
  case.
\item Analogues of the direct applications of
  Section~\ref{ssec:direct} were proved first, essentially, by
  Bruggeman~\cite{bruggeman} (analogue of Proposition~\ref{pr:rp} for
  Maass forms, where the Ramanujan-Petersson conjecture is not yet
  known); analogues of Proposition~\ref{pr:strong} are due to
  Sarnak~\cite{sarnak} (Maass forms) and Serre (holomorphic forms),
  though both counted the cusp forms with the natural weight $1$, and
  used the trace formula instead of the Petersson formula
  (correspondingly, their limiting distributions was different: at $p$
  they obtained the Plancherel measure for the unramified principal
  series of $\GL(2,\Q_p)$ with trivial central character).
\item Computations tantamount to working with the twisted weight
  $\omega'_f$ are also classical (in particular, the computation of
$$
\sum_{f}{\omega'_f \lambda_f(m)}
$$
for a fixed $m$ is a special case of the first moment computation
in~\cite{rankinselberglevel}, in the case where the level goes to
infinity and the Rankin-Selberg convolution is the weight $1$ theta
series with $L$-function $\zeta(s)L(s,\chi_4)$).
\item The analogue of B\"ocherer's conjecture for $\classical_k^*$ is
  the famous formula of Waldspurger~\cite{walds} which relates the
  value $L(f \times \chi_d, \frac{1}{2})$ for $f\in \classical_k^*$ to
  the squares of Fourier coefficients of modular forms of
  half-integral weight. However, these special values do not appear in
  the standard weights used for averaging $L$-functions. However, a
  weighted averaged version of Waldspurger's formula was proved by
  Iwaniec~\cite{iwaniec-walds} using identities for Kloosterman sums,
  and this may be considered as somewhat analogue to our
  Theorem~\ref{th-average-l-functions}.
\item The Saito-Kurokawa forms have no analogue in $\classical_k^*.$
  Indeed, all cusp forms of $\GL(2)$ (or $\GL(n)$) are expected to
  satisfy the Ramanujan-Petersson conjecture; for forms in
  $\classical_k^*$, this is a theorem of Deligne~\cite{deligne}. On
  the other hand, Saito-Kurokawa forms do not satisfy the generalized
  Ramanujan conjecture; this is due to the fact that they are CAP
  representations~\cite{ps-sk}.

\end{itemize}

\def\cprime{$'$} \newcommand{\noopsort}[1]{}

\end{document}